\documentclass{article}

\PassOptionsToPackage{numbers, compress,sort}{natbib}
\usepackage[preprint]{neurips_2026}
\usepackage[utf8]{inputenc}
\usepackage[T1]{fontenc}
\usepackage{hyperref}
\usepackage{hyperref}
\usepackage{url}
\usepackage{booktabs}
\usepackage{microtype}

\usepackage{amsmath, amssymb, amsfonts, caption, enumitem, graphicx, mathtools, nicefrac, comment, subcaption, wrapfig} 

\usepackage[dvipsnames]{xcolor}
\hypersetup{colorlinks, linkcolor={red!50!black}, citecolor={blue!50!black}, urlcolor={ForestGreen}}

\usepackage{algorithm, algorithmic}

\usepackage{mathabx}
\usepackage{tabularx}

\allowdisplaybreaks
\usepackage[capitalize]{cleveref}
\crefformat{equation}{(#2#1#3)}
\crefformat{condition}{Condition #2#1#3}
\crefrangeformat{equation}{(#3#1#4)--(#5#2#6)}
\crefrangeformat{condition}{Conditions #3#1#4--#5#2#6}
\crefrangeformat{lemma}{Lemmas #3#1#4--#5#2#6}
\crefrangeformat{proposition}{Propositions #3#1#4--#5#2#6}
\crefrangeformat{algorithm}{Algorithms #3#1#4--#5#2#6}
\crefrangeformat{section}{Sections #3#1#4--#5#2#6}
\crefrangeformat{assumption}{Assumptions #3#1#4--#5#2#6}
\crefformat{assumption}{Assumption #2#1#3}

\usepackage{amsthm}
\newtheorem{theorem}{Theorem}
\newtheorem{proposition}{Proposition}
\newtheorem{lemma}{Lemma}

\newtheorem{assumption}{Assumption}

\def\beq{\begin{equation}}
\def\eeq{\end{equation}}

\def\fnote#1{\footnote}

\newcommand{\x}{{\mathrm{x}}}
\newcommand{\y}{{\mathrm{y}}}

\def\E{{\mathbb{E}}}

\def\R{{\mathbb{R}}}

\def\cA{{\cal A}}

\def\cC{{\cal C}}
\def\cD{{\cal D}}

\def\cF{{\cal F}}

\def\cO{{\cal O}}

\def\cU{{\cal U}}

\def\cX{{\cal X}}
\def\cY{{\cal Y}}

\newcommand{\bbR}{\mathbb{R}}

\newcommand{\CX}{\mathcal{X}}
\newcommand{\CY}{\mathcal{Y}}

\DeclareMathOperator*{\argmin}{arg\,min}
\DeclareMathOperator*{\argmax}{arg\,max}

\DeclareMathOperator{\Tr}{Tr}

\DeclareMathOperator{\conv}{conv}

\DeclareMathOperator{\one}{\mathbf{1}}

\DeclarePairedDelimiter\abs{\lvert}{\rvert}%

\makeatletter
\let\oldabs\abs
\def\abs{\@ifstar{\oldabs}{\oldabs*}}
\makeatother

\DeclareMathOperator{\PO}{PO}
\DeclareMathOperator{\LMO}{LMO}
\newcommand{\Gap}{\mathsf{Gap}}
\newcommand{\gap}{\mathsf{gap}}
\newcommand{\op}{\mathrm{op}}

\newcommand{\myoverline}[1]{\mkern 1.75mu\overline{\mkern-1.75mu#1\mkern-1.75mu}\mkern 1.75mu}

\RequirePackage{silence}
\title{Projection-Free Algorithms for Nonsmooth \\ Stochastic Convex-Concave Saddle-Point Problems}

\author{
  Khanh-Hung Giang-Tran \\
  Cornell University \\ \texttt{tg452@cornell.edu}
  \And
  Soroosh Shafiee \\
  Cornell University \\
  \texttt{shafiee@cornell.edu} 
}

\begin{document}

\maketitle

\begin{abstract}
    We study nonsmooth convex-concave saddle-point problems over compact convex sets, assuming access to stochastic subgradients of the payoff function. We develop \emph{single-loop projection-free} algorithms that use linear minimization oracles over the primal and dual domains. Unlike prior projection-free approaches that rely on smoothing, our methods are purely subgradient-based and handle nonsmoothness directly. This design makes the framework modular: when a linear minimization oracle is not used over a specific domain, the corresponding update can be replaced by a standard projection oracle without changing the single-loop structure. Thus, our framework covers both fully projection-free and hybrid oracle configurations. We prove \emph{anytime} strong saddle gap guarantees under standard unbiased stochastic oracle assumptions with bounded variance. Our algorithms achieve an $O(\epsilon^{-2})$ iteration complexity, matching projected stochastic subgradient methods. We also provide matching lower bounds for the corresponding oracle models, establishing minimax optimality. Our results show that, for nonsmooth stochastic minimax problems, the computational advantages of linear minimization oracles need not come at the expense of statistical or oracle efficiency.
    %We study nonsmooth convex-concave saddle-point problems over compact convex sets, assuming access to stochastic subgradients of the payoff function. To solve these problems, we develop \emph{single-loop projection-free} algorithms that utilize linear minimization oracles over the primal and dual domains. Unlike prior projection-free approaches for nonsmooth optimization that rely on smoothing techniques, our methods are purely subgradient-based and handle nonsmoothness directly. This subgradient-based design makes the framework modular. Specifically, when a linear minimization oracle is not employed over a specific domain, the corresponding update can be replaced by a standard projection oracle without changing the overall single-loop structure. As a result, our framework covers both fully projection-free and hybrid oracle configurations. We prove \emph{anytime} strong saddle-gap guarantees under standard unbiased stochastic oracle assumptions with bounded second moments. Our algorithms achieve an $O(\epsilon^{-2})$ iteration complexity, matching the optimal rate of projected stochastic subgradient methods. We complement these upper bounds with matching lower bounds for the corresponding oracle models, thereby establishing the minimax optimality of our rates. Our results show that, for nonsmooth stochastic minimax problems, the computational advantages of linear minimization oracles need not come at the expense of statistical or oracle efficiency.
\end{abstract}

\section{Introduction}
\label{sec:introduction}

Projected subgradient descent is a fundamental approach for constrained nonsmooth convex optimization \citep{shor1985minimization,bertsekas1999nonlinear}. When projections onto the feasible set are computationally tractable, these methods are minimax optimal \citep{nemirovski1983problem,nesterov2013introductory}, reaching an $\epsilon$-approximate solution in $O(\epsilon^{-2})$ iterations. This optimality extends to the stochastic regime, where projected stochastic subgradient descent maintains the same iteration complexity \citep{robbins1951stochastic,nemirovski2009robust}. However, in many large-scale applications, the projection step becomes the primary computational bottleneck, necessitating the use of projection-free alternatives such as the Frank-Wolfe (FW) algorithm \citep{jaggi2013revisiting,braun2022conditional}.
While projection-free methods are well-established for smooth problems, their application to the nonsmooth setting is a more recent development. Notably, \citet{thekumparampil2020optimal,thekumparampil2020projection} and \citet{asgari2022projection,asgari2024nonsmooth} have introduced nonsmooth FW variants that achieve an iteration complexity of $O(\epsilon^{-2})$, matching the minimax optimal rate~\citep{lan2013complexity}.

This work addresses the nonsmooth convex-concave saddle-point problem of the form
\begin{equation}
\label{eq:minimax}
    \min_{x \in \cX} \ \max_{y \in \cY} \ f(x,y),
\end{equation}
where we assume access only to unbiased stochastic subgradients of $f$.
For this class of problems, projected stochastic subgradient descent-ascent is known to achieve the optimal convergence rate of $O(\epsilon^{-2})$ \citep{nedic2009subgradient,nemirovski2009robust}.
However, the landscape of projection-free methods for nonsmooth saddle-point problems is significantly less understood. This paper provide a comprehensive analysis of \emph{single-loop}, linear minimization oracle (LMO)-based algorithms for~\eqref{eq:minimax}, establishing anytime convergence guarantees. Our proposed framework is flexible, utilizing an LMO over $\cX$, $\cY$, or both. In instances where an LMO is not employed for a specific domain, we utilize a standard projection oracle (PO).

\subsection{Contributions}
Our main contributions are as follows. 
\begin{enumerate}[label=$\diamond$,leftmargin=2em]
    \item Building upon the subgradient-based framework introduced in \citep{asgari2022projection,asgari2024nonsmooth}, we propose a novel algorithmic update for saddle-point problems that employs a non-trivial mechanism to derive the LMO direction. This approach directly addresses objective nonsmoothness \emph{without} requiring smoothing techniques. Furthermore, the update rule is designed for modularity, allowing it to be seamlessly integrated with a standard PO when an LMO is not employed for a specific variable. This allows us to rigorously analyze three distinct oracle configurations: LMO-LMO (fully projection-free), LMO-PO, and PO-LMO (hybrid). This flexibility enables practitioners to select the most efficient oracle for each constraint set $(\cX, \cY)$.
    
    \item We establish \emph{anytime} convergence guarantees for all proposed algorithmic variants. This is of significant practical importance as it eliminates the need to pre-specify a target accuracy~$\epsilon$ or a total iteration horizon $T$ to tune parameters. We prove that our algorithms achieve the optimal $O(\epsilon^{-2})$ iteration complexity, matching the theoretical performance of standard projected stochastic subgradient descent-ascent methods \citep{nesterov2013introductory, nemirovski1983problem, guzman2015lower}.

    \item We complement our upper bounds with matching lower bounds. In particular, we prove an $\Omega(\epsilon^{-2})$ iteration complexity lower bound for any algorithm based on the LMO-LMO oracle model, as well as for any method operating under the LMO-PO oracle model. By strong duality, the same lower bound immediately extends to the PO-LMO setting. Hence, the rates achieved by our algorithms are \emph{minimax optimal} in all three oracle regimes.

    \item Our results substantially strengthen the current theory of projection-free minimax optimization. In particular, even in the \emph{smooth deterministic single-loop} setting, previous projection-free methods were only known to achieve slower rates \citep{boroun2023projection,giang2026projection}. By contrast, we obtain the optimal $O(\epsilon^{-2})$ rate in the more general \emph{nonsmooth stochastic} setting. This closes the convergence-rate gap between projection-free and projection-based methods for minimax optimization, showing that the computational advantages of LMO-based updates need not come at the expense of statistical or oracle complexity.
\end{enumerate}

\subsection{Related Works}

\paragraph{Nonsmooth Frank-Wolfe Algorithms.} 
Constrained minimax problems of the form \eqref{eq:minimax} can be viewed as instances of nonsmooth convex optimization. This is because the primal function $p(x) := \max_{y \in \cY} f(x, y)$ is generally nonsmooth, even when $f$ is smooth. The classic FW algorithm, designed for smooth objectives, cannot be directly applied here. 
\citet{white1993extension} first addressed nonsmooth optimization using an LMO. Since then, only a few studies have followed, and earlier works often rely on strict assumptions \citep{ravi2019deterministic,cheung2017nonsmooth}. More recently, \citet{hazan2012projection} introduced the Online Frank-Wolfe algorithm, obtaining an $O(T^{-3/4})$ regret bound for general online convex optimization. This translates to an $O(\epsilon^{-4})$ rate for the offline nonsmooth convex setting. Under an additional smoothness assumption, \citet{hazan2020faster} improved this complexity to $O(\epsilon^{-3})$ by using a Follow-the-Perturbed-Leader approach as a smoothing technique. \citet{lan2013complexity} used randomized smoothing in the deterministic setting to achieve an $O(\epsilon^{-2})$ rate.

A complementary approach addresses nonsmoothness through the composite structure of the objective. \citet{lan2016conditional} and \citet{lan2016gradient} proposed the conditional gradient sliding algorithm, which decouples gradient and LMO computations. Using this sliding framework with Moreau-Yosida smoothing, \citet{thekumparampil2020projection} proposed a \emph{double-loop} algorithm that achieves $\widetilde O(\epsilon^{-2})$ LMO and subgradient oracle calls. \citet{thekumparampil2020optimal} then extended this framework to a \emph{single-loop} procedure for nonsmooth minimization. However, both algorithms rely on smoothing techniques that require parameters to be selected a priori. Consequently, these algorithms are not anytime procedures.
Rather than relying on smoothing, \citet{asgari2022projection} recently proposed a subgradient-based projection-free algorithm for nonsmooth convex optimization. This method attains an $O(\epsilon^{-2})$ convergence rate using a single LMO call per iteration, serving as a projection-free counterpart to projected subgradient descent. This was extended by \citet{asgari2024nonsmooth} to handle general convex functional inequality constraints while preserving the $O(\epsilon^{-2})$ complexity, consistent with existing lower bounds \citep{lan2013complexity}. A key algorithmic feature we share with the nonsmooth literature \citep{lan2016conditional,tao2019strength,thekumparampil2020optimal,thekumparampil2020projection,asgari2022projection,asgari2024nonsmooth,lu2023projection,grimmer2024radialI,grimmer2024radialII} is the use of subgradients \emph{outside the feasible set}. Besides, similar to the approach in \citep{asgari2022projection,asgari2024nonsmooth}, our analysis is subgradient-based rather than dependent on smoothing techniques.

\paragraph{LMO-Based Methods for Minimax Problems.}
Earlier works focused on special primal-dual Lagrangian-based formulations and established $O(\epsilon^{-2})$ rates in the smooth setting \citep{harchaoui2015conditional,yurtsever2019conditional,locatello2019stochastic,lan2021conditional}.
Subsequent work extended this to smooth convex-strongly concave problems via stochastic \emph{double-loop} methods \citep{chen2020efficient} and to smooth convex-concave problems via \emph{multi-loop} Follow-the-Perturbed-Leader approach \citep{suggala2020follow}, both recovering the $O(\epsilon^{-2})$ rate.
\citet{gidel2017frank} pioneered the study of \emph{single-loop} algorithms with linear convergence for smooth strongly convex-strongly concave problems with interior solutions, as well as for smooth convex-concave problems over strongly convex sets with lower-bounded gradients or polytope feasible set.
\citet{kolmogorov2021one} developed a one-sided method for the bilinear case, achieving $O(\epsilon^{-1})$ rate.

A separate line of work considers the smooth nonconvex-concave setting, whose guarantees translate to the convex-concave case but at inferior rates since convexity is not exploited.
\citet{nouiehed2019solving} proposed a multi-loop one-sided projection-free method achieving an $O(\epsilon^{-3.5})$ iterations. 
\citet{boroun2023projection} and \citet{giang2026projection} subsequently developed single-loop projection-free analyses for nonconvex-concave problems under various conditions, with rates ranging from $O(\epsilon^{-6})$ to $O(\epsilon^{-2})$ depending on the setting.
To the best of our knowledge, this is the first work to address nonsmooth stochastic convex-concave minimax problems while achieving the optimal $O(\epsilon^{-2})$ rate.

\paragraph{LMO-Based Methods for Variational Inequalities.}
Smooth convex-concave minimax problems are an instance of monotone variational inequality problems.
\citet{hammond1984solving} showed that a single-loop FW algorithm with an open-loop stepsize converges \emph{asymptotically} for monotone variational inequalities over strongly convex sets, and this was recently extended to general convex compact sets in \citep{hough2026asymptotic}.
While several works have established non-asymptotic complexity bounds via \emph{nested multi-loop} schemes, such as semi-proximal mirror-prox methods \citep{he2015semi} and extra gradient methods \citep{baghbadorani2025frank}, single-loop algorithms have required either specific geometries or stronger oracles to achieve competitive rates.
Specifically, \citet{cox2014dual} and \citet{juditsky2016solving} proved an $O(\epsilon^{-2})$ ergodic convergence rate under the accessibility of Fenchel-type representation of the payoff function and its corresponding first-order information. 
Moreover, \citet{chen2024last} proved an $O(\epsilon^{-2})$ last-iterate convergence rate for a single-loop \emph{generalized} FW method, but their approach relies on a stronger oracle rather than a pure LMO.
In contrast, this paper considers nonsmooth stochastic minimax problems and develops single-loop purely LMO-based algorithms, enjoying the same rate.

\paragraph{PO-Based Methods.} 
Many algorithms exist for solving monotone variational inequalities and smooth convex-concave saddle-point problems of the form~\eqref{eq:minimax}.
In the deterministic setting, the extragradient method \citep{korpelevich1976extragradient}, Popov's algorithm \citep{popov1980modification}, and Tseng's forward-backward-forward scheme \citep{tseng2000modified} all achieve an $O(\epsilon^{-1}$) rate, as does the mirror-prox algorithm of \citet{nemirovski2004prox}.
This rate is optimal \citep{ouyang2021lower} and is also attained by modern variants including reflected gradient methods \citep{malitsky2015projected,malitsky2020forward}, optimistic gradient descent-ascent \citep{daskalakis2018training}, dual extrapolation \citep{nesterov2007dual}, and adaptive methods \citep{malitsky2020golden,alacaoglu2023beyond}.
For structured problems such as Lagrangian-based payoff functions for affinely constrained problems, specialized methods can exploit additional structure to obtain faster rates \citep{chambolle2011first,nesterov2005smooth}.
Moving to the nonsmooth setting, \citet{nedic2009subgradient} established an $O(\epsilon^{-2})$ rate for projected subgradient descent-ascent, and \citet{nemirovski2009robust} achieved the same rate under additional stochasticity.
We develop algorithms based on various combinations of LMO and PO for nonsmooth stochastic convex-concave problems and establish the optimal $O(\epsilon^{-2})$ rate under this more challenging setting. 

\paragraph{Lower Bound Complexities.} Lower bounds for first-order methods have been extensively studied in the optimization literature. For nonsmooth convex minimization, the $O(\epsilon^{-2})$ complexity of subgradient methods is known to be optimal \citep{nesterov2013introductory, nemirovski1983problem, guzman2015lower}. For smooth convex minimization, the matching lower bound for gradient-based methods was established by \citet{nesterov1983method}, while \citet{lan2013complexity} derived analogous lower bounds for conditional gradient methods. 
More recently, lower bounds under additional strong convexity of the feasible set were studied by \citet{halbey2026lower} and \citet{grimmer2026lower}.
In the stochastic setting, optimal lower bounds for convex and nonconvex problems have been established in \citep{agarwal2009information, foster2019complexity}. For our minimax optimization problem, we show that our $O(\epsilon^{-2})$ upper bounds match their corresponding lower bounds, establishing that these rates are minimax optimal.

\subsection{Notation and Outline}
For a compact convex set $\cU$, the LMO at a given direction $d$ and the PO at a point $u$ are defined as
\begin{align*}
    \LMO_{\mathcal{U}}(d) \in \argmin_{v \in \mathcal{U}} ~ d^\top v 
    \qquad \text{and} \qquad 
    \PO_{\cU}(u) := \argmin_{v \in \cU} \| u - v \|_2^2.
\end{align*}
Convergence results in the main text are reported using standard big-$O$ and $\widetilde O$ notations, where the latter suppresses logarithmic factors for clarity. Detailed proofs, including explicit convergence constants, as well as implementation details are provided in the Appendix.
Section~\ref{sec:algorithms} presents our proposed algorithms and derives the key recursions underlying each variant. Section~\ref{sec:convergence} establishes their convergence guarantees, and Section~\ref{sec:lower-bounds} provides lower bounds for any algorithm relying on combinations of LMO and PO updates. Section~\ref{sec:numerical} reports numerical experiments comparing our methods against state-of-the-art baselines.

\section{Proposed Algorithms}
\label{sec:algorithms}

We begin by stating our assumptions on the problem structure and the stochastic subgradient oracle.
\medskip
\begin{assumption}
\label{ass:extended}
    The payoff function $f: \bbR^n \times \bbR^m \to \bbR$ and the feasible sets $\cX \subset \R^n$ and $\cY \subset \R^m$ satisfy the followings. 
    \begin{enumerate}[label=(\roman*),leftmargin=2em]
        \item\label{ass:X-compact-extend} The domain $\CX$ is a compact convex set and there exists another convex compact set $\myoverline{\CX} \supseteq \CX$ with $$R_{\cX} := \max_{x \in \CX, \bar{x} \in \myoverline{\CX}} \|x-\bar{x}\|_2>0.$$ \\[-2.5em]
        \item\label{ass:Y-compact-extend} The domain $\CY$ is a compact convex set and there exists another convex compact set $\myoverline{\CY} \supseteq \CY$ with $$R_{\cY} := \max_{y \in \CY, \bar{y} \in \myoverline{\CY}} \|y-\bar{y}\|_2>0.$$ \\[-2.5em]
        
        \item \label{ass:fx-convex-extend} The function $f_y(\cdot) := f(\cdot,y)$ is convex on $\myoverline{\CX}$.

        \item \label{ass:fy-convex-extend} The function $(-f)_x(\cdot) := -f(x,\cdot)$ is convex on $\myoverline{\CY}$.
    \end{enumerate}
\end{assumption} 

\begin{assumption}\label{ass:stochastic-oracle}
    There exists a stochastic oracle $\cO_f$ such that given any $(x,y) \in  \myoverline{\CX}\times  \myoverline{\CY}$ and $(g^\x,g^\y) = \cO_f(x,y)$, it holds that
    \begin{enumerate}[label=(\roman*),leftmargin=2em] 
        \item\label{ass:unbiased} $\E[(g^\x, g^\y)] \in \partial f_y(x) \times \partial (-f)_x(y) $.
         \item\label{ass:bounded-variance} There exist finite $G_\x,G_\y \!>\! 0$ independent of $(x,y)$ such that $\E[\|g^\x\|_2^2] \!\leq\! G_\x^2$ and $\E[\|g^\y\|_2^2] \!\leq\! G_\y^2$.
    \end{enumerate}
\end{assumption} 

\begin{table*}[!b]
\centering
\begin{minipage}[t]{0.33\textwidth}
\vspace{-1.8em}
\begin{algorithm}[H]
    \small
    \caption{\texttt{Nonsmooth LMO-LMO}}
    \label{alg:lmo-lmo}
    \begin{algorithmic}
        \REQUIRE Parameters $\alpha_t, \beta_t, \eta_t, \tau_t \!>\! 0$ \\[1ex]
        \hspace{-1.25em} \textbf{Initialize:} $v_1 \!=\! x_1 \!\in\! \cX$, $\bar x_0 \!=\! \lambda_1 \!=\! 0$ \\
        \hspace{3.1em} $u_1 \!=\! y_1 \!\in\! \cY$, $\bar y_0 \!=\! \mu_1 \!=\! 0$ \vspace{0.5ex}
        \FOR{$t = 1, \dots, T$} \vspace{0.5ex}
            \STATE $(g^\x_t, g^\y_t) = \mathcal{O}_f(v_t, u_t)$ \vspace{0.5ex}
            \STATE $\bar x_t = \frac{t-1}{t} \bar x_{t-1} + \frac{1}{t} x_t $
            \STATE $v_{t+1} = \PO_{\myoverline{\CX}}\bigl(v_t - \frac{\lambda_t + g^\x_t}{\alpha_t}\bigr)$ \vspace{0.5ex}
            \STATE $x_{t+1} = \LMO_{\mathcal{X}}(-\lambda_t)$ \vspace{0.5ex}
            \STATE $\lambda_{t+1} \!=\! \eta_{t\!+\!1} \big(\tfrac{\lambda_t}{\eta_t} \!+\! v_{t\!+\!1} \!-\! x_{t\!+\!1} \big)$ \vspace{0.5ex}
            \STATE $\bar y_t = \frac{t-1}{t} \bar y_{t-1} + \frac{1}{t} y_t $
            \STATE $u_{t+1} = \PO_{\myoverline{\CY}}\bigl(u_t - \frac{\mu_t + g^\y_t}{\beta_t}\bigr)$ \vspace{0.5ex}
            \STATE $y_{t+1} = \LMO_{\mathcal{Y}}(-\mu_t)$ \vspace{0.5ex}
            \STATE $\mu_{t+1} \!=\! \tau_{t\!+\!1} \big(\tfrac{\mu_t}{\tau_t} \!+\! u_{t\!+\!1} \!-\! y_{t\!+\!1} \big)$ \vspace{0.5ex}
        \ENDFOR
        \vspace{0.5ex}
        \ENSURE $\bar x_T, \bar y_T$
    \end{algorithmic}
\end{algorithm}
\end{minipage}
\hfill
\begin{minipage}[t]{0.33\textwidth}
\vspace{-1.8em}
\begin{algorithm}[H]
    \small
    \caption{\texttt{Nonsmooth LMO-PO}}
    \label{alg:lmo-po}
    \begin{algorithmic}
        \REQUIRE Parameters $\alpha_t, \gamma_t, \eta_t > 0$ \\[1ex]
        \hspace{-1.25em} \textbf{Initialize:} $v_1 \!=\! x_1 \!\in\! \cX$, $\bar x_0 \!=\! \lambda_1 \!=\! 0$ \\
        \hspace{3.1em} $y_1 \!\in\! \cY$, $\bar y_0 \!=\! 0$ \vspace{0.5ex}
        \FOR{$t = 1, \dots, T$} \vspace{0.5ex}
            \STATE $(g^\x_t, g^\y_t) = \mathcal{O}_f(v_t, u_t)$ \vspace{0.5ex}
            \STATE $\bar x_t = \frac{t-1}{t} \bar x_{t-1} + \frac{1}{t} x_t $
            \STATE $v_{t+1} = \PO_{\myoverline{\CX}}\bigl(v_t - \frac{\lambda_t + g^\x_t}{\alpha_t}\bigr)$ \vspace{0.5ex}
            \STATE $x_{t+1} = \LMO_{\mathcal{X}}(-\lambda_t)$ \vspace{0.5ex}
            \STATE $\lambda_{t+1} \!=\! \eta_{t\!+\!1} \big(\tfrac{\lambda_t}{\eta_t} \!+\! v_{t\!+\!1} \!-\! x_{t\!+\!1} \big)$ \vspace{0.5ex}
            \STATE $\bar y_t = \frac{t-1}{t} \bar y_{t-1} + \frac{1}{t} y_t $ \vspace{3.5em}
            \STATE $y_{t+1} = \PO_{\mathcal{Y}}(y_t - \gamma_t g^\y_t)$ \vspace{0.5ex}
        \ENDFOR
        \vspace{0.5ex}
        \ENSURE $\bar x_T, \bar y_T$
    \end{algorithmic}
\end{algorithm}
\end{minipage}
\hfill
\begin{minipage}[t]{0.325\textwidth}
\vspace{-1.8em}
\begin{algorithm}[H]
    \small
    \caption{\texttt{Nonsmooth PO-LMO}}
    \label{alg:po-lmo}
    \begin{algorithmic}
        \REQUIRE Parameters $\rho_t, \beta_t, \tau_t > 0$ \\[1ex]
        \STATE  \textbf{Initialize:} $x_1 \in \cX$ \\
        \hspace{4.15em} $u_1 = y_1 \in \cY$, $\mu_0 \!=\! 0$ \vspace{0.5ex}
        \FOR{$t = 1, \dots, T$} \vspace{0.5ex}
            \STATE $(g^\x_t, g^\y_t) = \mathcal{O}_f(x_t, u_t)$ \vspace{0.5ex}
            \STATE $\bar x_t = \frac{t-1}{t} \bar x_{t-1} + \frac{1}{t} x_t $ \vspace{3.7em}
            \STATE $x_{t+1} = \PO_{\mathcal{X}}(x_t - \rho_t g^\x_t)$ \vspace{0.5ex}
            \STATE $\bar y_t = \frac{t-1}{t} \bar y_{t-1} + \frac{1}{t} y_t $
            \STATE $u_{t+1} = \PO_{\myoverline{\CY}}\bigl(u_t - \frac{\mu_t + g^\y_t}{\beta_t}\bigr)$ \vspace{0.5ex}
            \STATE $y_{t+1} = \LMO_{\mathcal{Y}}(-\mu_t)$ \vspace{0.5ex}
            \STATE $\mu_{t+1} \!=\! \tau_{t\!+\!1} \big(\tfrac{\mu_t}{\tau_t} \!+\! u_{t\!+\!1} \!-\! y_{t\!+\!1} \big)$ \vspace{0.5ex}
        \ENDFOR
        \vspace{0.5ex}
        \ENSURE $\bar x_T, \bar y_T$
    \end{algorithmic}
\end{algorithm}
\end{minipage}
\end{table*}

Assumption~\ref{ass:stochastic-oracle} is standard in stochastic programming where $f(x,y) := \E_{z}[F(x,y,z)]$, in which the oracle $\cO_f$ can be implemented by sampling $z$ and returning $(g^\x, g^\y) \in \partial_\x F(x,y,z) \times \partial_\y (-F(x,y,z))$. The validity of this construction follows from \cref{ass:extended}, which ensures the interchange of expectation and subdifferentiation \citep[Section~3.2]{nemirovski2009robust}.

We now describe our three proposed algorithms, summarized in \cref{alg:lmo-lmo,alg:lmo-po,alg:po-lmo}.
All three share a common subroutine for computing LMO-based directions, which we describe for the primal variable~$x$; the dual counterpart is symmetric.

The method maintains a feasible point $x_t \in \cX$ and an auxiliary point $v_t \in \myoverline{\CX}$, where $\myoverline{\CX}$ is a simple enclosing set of $\cX$. This set is chosen so that projection onto $\myoverline{\CX}$ is \emph{computationally cheap}; typical choices include boxes and Euclidean balls.
At iteration $t$, the algorithm uses the current accumulated direction $\lambda_t$ together with the stochastic subgradient $g^\x_t$ to update the auxiliary point $v_{t+1}$.
The next feasible point is then obtained through a linear minimization oracle over the original domain $x_{t+1}=\LMO_{\cX}(-\lambda_t)$.
Finally, $\lambda_{t+1}$ is updated from the discrepancy $v_{t+1}-x_{t+1}$, so that $\lambda_t$ accumulates information about the mismatch between the auxiliary and feasible sequences. Since the algorithm avoids projection onto the potentially complicated feasible set $\cX$ and instead relies on an LMO over $\cX$, with only a cheap projection onto the simple enclosing set $\myoverline{\CX}$, we refer to it as \emph{projection-free}.

Algorithm~\ref{alg:lmo-lmo} applies this procedure to both primal and dual variables.
Algorithms~\ref{alg:lmo-po} and~\ref{alg:po-lmo} are hybrid variants.
That is, Algorithm~\ref{alg:lmo-po} uses the LMO-based update for $x$ and a standard projected subgradient step for~$y$, while Algorithm~\ref{alg:po-lmo} does the reverse. All three algorithms output the running average $\bar x_T := \frac{1}{T} \sum_{t \in [T]} x_t$ and $\bar y_T := \frac{1}{T} \sum_{t \in [T]} y_t$ as the approximate solution, which is standard practice for nonsmooth stochastic methods to ensure convergence of gap functions.

We assess the quality of the output $(\bar x_T, \bar y_T)$ via the \emph{strong saddle gap}, defined as
\begin{align*}
    \Gap(\bar x_T, \bar y_T) := \E \left[\max_{y \in \cY} f(\bar{x}_T, y) - \min_{x \in \cX} f(x, \bar{y}_T) \right].
\end{align*}
By Jensen's inequality, the \emph{weak saddle gap}, defined as
\begin{align*}
    \gap(\bar x_T, \bar y_T) := \max_{y \in \cY} \E[f(\bar{x}_T, y)] - \min_{x \in \cX} \E[f(x, \bar{y}_T)],
\end{align*}
satisfies
\begin{align*}
    \gap(\bar x_T, \bar y_T) \leq \Gap(\bar x_T, \bar y_T).
\end{align*}
Therefore, all subsequent guarantees carry over to the weak saddle gap as well, which is the metric used in~\citep{nemirovski2009robust}.
We also note that the strong saddle gap admits a natural generalization to variational inequalities, as discussed in \citep{juditsky2011solving}. 

The following proposition isolates the key recursive inequalities underlying the LMO-based updates, which serve as the common building block in the convergence analyses of~\cref{alg:lmo-lmo,alg:lmo-po,alg:po-lmo}. 
\medskip
\begin{proposition}
\label{prop:generic-lmo}
    Suppose Assumptions~\ref{ass:extended} and \ref{ass:stochastic-oracle} hold.
    \begin{enumerate}[label=(\roman*),leftmargin=2em]
        \item Let $\{v_t,x_t,\lambda_t\}_{t \geq 1}$ be the primal sequences generated by Algorithms~\ref{alg:lmo-lmo} or \ref{alg:lmo-po} with parameters satisfying $\alpha_t \leq \alpha_{t+1}$ and $\eta_t \geq \eta_{t+1}$ for all $t \geq 1$. 
        Define the dual evaluation sequence $\{z_t\}_{t \geq 1} = \{ u_t \}_{t \geq 1}$ for \cref{alg:lmo-lmo} and $\{z_t\}_{t \geq 1} = \{ y_t \}_{t \geq 1}$ for \cref{alg:lmo-po}.
        Then, for every $T \geq 2$, we have 
        \begin{equation}
        \label{eq:generic-primal-lmo}
        \begin{aligned}
            &\E\left[\frac{\eta_{T-1}}{2} \left\|\frac{\lambda_T}{\eta_T}\right\|_2^2 + \max_{x \in \cX} \; \sum_{t = 1}^{T} \big( f(v_t,z_t) - f(x,z_t) \big) \right]  \\
            &\leq 2\left(\alpha_{T-1} +\sum_{t = 1}^{T-1}\eta_t\right) R_{\cX}^2+\frac{G_\x^2}{2\alpha_{T-1}}+\sum_{t = 1}^{T-1}\frac{G_\x^2}{2 \alpha_t}+ 2G_\x R_{\cX} T^{1/2}. 
        \end{aligned}  
        \end{equation}
    
        \item Let $\{u_t,y_t,\mu_t\}_{t \geq 1}$ be the dual sequences generated by Algorithms~\ref{alg:lmo-lmo} or \ref{alg:po-lmo} with parameters satisfying $\beta_t \leq \beta_{t+1}$ and $\tau_t \geq \tau_{t+1}$ for all $t \geq 1$.
        Define the primal evaluation sequence $\{w_t\}_{t \geq 1} = \{ v_t \}_{t \geq 1}$ for \cref{alg:lmo-lmo} and $\{w_t\}_{t \geq 1} = \{ x_t \}_{t \geq 1}$ for \cref{alg:po-lmo}.
        Then, for every $T \geq 2$, we have 
        \begin{equation}
        \label{eq:generic-dual-lmo}
        \begin{aligned}
            &\E\left[\frac{\tau_{T-1}}{2} \left\|\frac{\mu_T}{\tau_T}\right\|_2^2 + \max_{y \in \cY} \sum_{t = 1}^{T}  \big( f(w_t, y) - f(w_t, u_t) \big) \right]\\
            &\leq 2\left(\beta_{T-1} +\sum_{t = 1}^{T-1}\tau_t\right) R_{\cY}^2+\frac{G_\y^2}{2\beta_{T-1}}+\sum_{t = 1}^{T-1}\frac{G_\y^2}{2 \beta_t}+ 2G_\y R_{\cY}T^{1/2}.
        \end{aligned}
        \end{equation}
    \end{enumerate}
\end{proposition}

Although the proposition is stated for both primal and dual updates, inequalities~\eqref{eq:generic-primal-lmo} and~\eqref{eq:generic-dual-lmo} are \emph{symmetric} in structure. 
This symmetry is by design. Specifically, strong duality of the minimax problem~\eqref{eq:minimax} allows us to treat the primal and dual updates on equal footing, leading to algorithms with an identical structure for both variables. We therefore only provide the primal proof in the appendix, as the dual proof follows by an analogous argument.

The next proposition isolates the key recursive inequalities underlying the PO-based updates, which serve as the common building block in the convergence analyses of~\cref{alg:lmo-po,alg:po-lmo}. 

\begin{proposition}
\label{prop:generic-po}
    Suppose Assumptions~\ref{ass:extended} and \ref{ass:stochastic-oracle} hold.
    \begin{enumerate}[label=(\roman*),leftmargin=2em]
        \item Let $\{x_t\}_{t \geq 1}$ be the primal sequence generated by Algorithm~\ref{alg:po-lmo} with parameters satisfying $\rho_t \geq \rho_{t+1}$ for all $t \geq 1$.
        Then, for every $T \geq 1$,
        \begin{align}
        \label{eq:generic-primal-po}
            \E\left[\max_{x \in \cX}\sum_{t =1}^T \big( f(x_t, u_t) - f(x, u_t) \big) \right] &\leq \frac{2R_{\cX}^2}{\rho_T}+\sum_{t = 1}^T \frac{\rho_t G_\x^2}{2} + 2G_\x R_{\cX}T^{1/2}.
        \end{align}
        
        \item Let $\{y_t\}_{t \geq 1}$ be the dual sequences generated by Algorithm~\ref{alg:lmo-po} with parameters satisfying $\gamma_t \geq \gamma_{t+1}$ for all $t \geq 1$. 
        Then, for every $T \geq 1$, 
        \begin{align}
        \label{eq:generic-dual-po}
            \E\left[\max_{y \in \cY}\sum_{t =1}^T \big( f(v_t,y)-f(v_t,y_t) \big) \right] &\leq \frac{2R_{\cY}^2}{\gamma_T}+\sum_{t = 1}^T \frac{\gamma_t G_\y^2}{2} + 2G_\y R_{\cY}T^{1/2}.
        \end{align}
    \end{enumerate}
\end{proposition}

\cref{prop:generic-lmo,prop:generic-po} provide all the building blocks needed for the convergence analyses of the three algorithms, each of which will follow by combining the appropriate pair of inequalities. As before, the primal and dual inequalities~\eqref{eq:generic-primal-po} and \eqref{eq:generic-dual-po} are symmetric in structure, so we provide only the proof of the primal part in the appendix. To keep the analysis unified, we also prove a slightly more general version of \cref{prop:generic-po} that accommodates the PO-PO combination.

\section{Convergence Guarantees}
\label{sec:convergence}

We now present the convergence guarantees of \cref{alg:lmo-lmo,alg:lmo-po,alg:po-lmo}.
\medskip
\begin{theorem} \label{theorem:convergence}
    Suppose Assumptions~\ref{ass:extended} and~\ref{ass:stochastic-oracle} hold. 
    \begin{enumerate}[label=(\roman*),leftmargin=2em]
        \item 
        \label{theorem:convergence-lmo-lmo}
        Let $\{x_t,y_t\}_{t \geq 1}$ be the sequence generated by \cref{alg:lmo-lmo} with parameters $\alpha_t = \Theta(t^{1/2})$, $\beta_t= \Theta(t^{1/2})$, $\eta_t =\Theta(t^{-1/2})$ and $\tau_t=\Theta(t^{-1/2})$ for any $t \geq 1$. 
        Then, for every $T \ge 2$,
        \begin{align*}
            \Gap(\bar{x}_T, \bar{y}_T) 
            = \E \left[ \max_{y \in \cY}f(\bar{x}_T,y) - \min_{x \in \cX}f(x,\bar{y}_T) \right] 
            \leq O(T^{-1/2}).
        \end{align*}
        \item 
        \label{theorem:convergence-lmo-po}
        Let $\{x_t,y_t\}_{t \geq 1}$ be the sequence generated by \cref{alg:lmo-po} with parameters $\alpha_t= \Theta(t^{1/2})$, $\gamma_t= \Theta(t^{-1/2})$ and $\eta_t =\Theta(t^{-1/2})$ for any $t \geq 1$. 
        Then, for every $T \ge 2$,
        \begin{align*}
            \Gap(\bar{x}_T, \bar{y}_T) 
            = \E \left[ \max_{y \in \cY}f(\bar{x}_T,y) - \min_{x \in \cX}f(x,\bar{y}_T) \right] 
            \leq O(T^{-1/2}).
        \end{align*}
        \item 
        \label{theorem:convergence-po-lmo}
        Let $\{x_t,y_t\}_{t \geq 1}$ be the sequence generated by \cref{alg:po-lmo} with parameters $\rho_t= \Theta(t^{-1/2})$, $\beta_t= \Theta(t^{1/2})$ and $\tau_t=\Theta(t^{-1/2})$ for any $t \geq 1$. 
        Then, for every $T \ge 2$,
        \begin{align*}
            \Gap(\bar{x}_T, \bar{y}_T) 
            = \E \left[ \max_{y \in \cY}f(\bar{x}_T,y) - \min_{x \in \cX}f(x,\bar{y}_T) \right] 
            \leq O(T^{-1/2}).
        \end{align*}
    \end{enumerate}
\end{theorem} 

The proof of \cref{theorem:convergence}, with explicit constants and parameter settings, is provided in the appendix, where it is broken into three separate theorems each handling one oracles combination. To keep the analysis unified, we also include a proof for the PO-PO combination, which yields an anytime implementation of the method of \citet{nemirovski2009robust} with the strong saddle gap guarantee, as opposed to the weak guarantee established in the original work.

We now position our contributions within the literature on projection-free minimax optimization. Across all three oracle configurations, the same advantages apply: by avoiding dual smoothing and exploiting the natural symmetry of the convex-concave saddle-point structure, we achieve a uniform $O(\epsilon^{-2})$ oracle complexity, substantially improving on prior work.

In the context of smooth deterministic minimax optimization problems, \citet{boroun2023projection} proposed the first single-loop LMO-LMO method for nonconvex-concave problems, and their rate easily translates to the convex-concave setting, yielding $\cO(\epsilon^{-6})$ complexity. However, their implementation is not anytime, requiring the total number of iterations or the target precision to be fixed in advance. Moreover, the algorithm is only applicable when the feasible set $\cY$ is strongly convex, as the stepsize selection crucially relies on this knowledge. \citet{giang2026projection} relax this requirement and propose an anytime implementation achieving the same $\cO(\epsilon^{-6})$ rate, which can be further improved to $\cO(\epsilon^{-5})$ when $\cY$ is additionally strongly convex. Nonetheless, both approaches are restricted to smooth deterministic problems and rely on a dual smoothing framework, which breaks the inherent symmetry of convex-concave saddle-point problems. As a result, they achieve these inferior rates even in the smooth deterministic setting. In contrast, our algorithm handles the more general nonsmooth stochastic setting, avoids smoothing entirely, and admits a naturally symmetric implementation that, as we will see, leads to the improved rate of $\cO(\epsilon^{-2})$.

Both \citet{boroun2023projection} and \citet{giang2026projection} also study the LMO-PO configuration, again restricted to smooth deterministic problems. The former achieves $\cO(\epsilon^{-4})$ without an anytime guarantee, while the latter removes this limitation at the same rate. Notably, the LMO-PO configuration outperforms its fully projection-free counterpart, because a single PO-based dual update yields a more accurate estimate of the gradient of the regularized primal function than a single LMO-based update, leading to a tighter control of the smoothing error. Yet despite this advantage, both methods remain hampered by the dual smoothing framework and its symmetry-breaking effect. Our algorithm, which operates in more general setting and avoids smoothing altogether, achieves $\cO(\epsilon^{-2})$ regardless.

The PO-LMO configuration was first studied in the context of smooth deterministic minimax optimization by \citet{giang2026projection}, where an anytime implementation is proposed. Unlike the LMO-PO case, replacing the primal PO update with an LMO update leads to a coarser approximation of the primal gradient, which explains why the rate degrades back to $\cO(\epsilon^{-6})$, matching the fully projection-free case. Our algorithm sidesteps this degradation entirely: by operating in the nonsmooth stochastic setting without smoothing and preserving the natural symmetry of the saddle-point structure, it achieves $\cO(\epsilon^{-2})$ across all three oracle configurations.

\section{Lower Bounds}
\label{sec:lower-bounds}

We now show that the upper bounds in \cref{theorem:convergence} are minimax optimal in their dependence on $T$, up to absolute constants. To this end, we introduce the generic LMO-LMO and LMO-PO procedures in \cref{alg:generic-lmo-lmo,alg:generic-lmo-po}, respectively. 
Throughout this section, lower bounds are stated in a worst-case oracle model. Thus, the adversary may choose the feasible sets $\cX,\cY$, the initial points $x_1,y_1$, the payoff function~$f$, and a valid tie-breaking rule for the LMO oracles. In particular, the adversary initializes $x_1$ and $y_1$ at extreme points and uses LMO oracles that always return extreme-point minimizers. 
For the LMO-PO class, only the primal variable is constrained by LMO calls, while the dual variable may be updated via any arbitrary map of historical first-order and value information, potentially incorporating projections. We omit the PO-LMO case, since the corresponding lower bound follows by swapping the primal and dual roles.

\begin{table*}[!h]
\centering
\begin{minipage}[t]{0.4\textwidth}
\vspace{-2em}
\begin{algorithm}[H]
    \small
    \caption{A generic \texttt{LMO-LMO} method}
    \label{alg:generic-lmo-lmo}
    \begin{algorithmic}
        \REQUIRE $ x_1 \in \cX,y_1 \in \cY$ \vspace{0.5ex}
        \FOR{$t = 1, \dots, T$} \vspace{0.5ex}
            \STATE Construct some $(p_t, q_t) \in \R^n \times \R^m$ \vspace{0.5ex}
            \STATE $x_{t+1} = \LMO_{\mathcal{X}}(p_t)$ \vspace{0.5ex}
            \STATE $y_{t+1} = \LMO_{\mathcal{Y}}(-q_t)$ \vspace{0.5ex}
        \ENDFOR
        \vspace{0.5ex}
        \ENSURE $\bar x_T \in \conv\left(\{x_1,\dots,x_{T+1}\}\right),$ \\ 
        \hspace{2.7em} $\bar y_T \in \conv\left(\{y_1,\dots,y_{T+1}\}\right)$
    \end{algorithmic}
\end{algorithm}
\end{minipage}
\hfill
\begin{minipage}[t]{0.55\textwidth}
\vspace{-2em}
\begin{algorithm}[H]
    \small
    \caption{A generic \texttt{LMO-PO} method}
    \label{alg:generic-lmo-po}
    \begin{algorithmic}
        \REQUIRE $ x_1 \in \cX,y_1 \in \cY$ \vspace{0.5ex}
        \FOR{$t = 1, \dots, T$} \vspace{0.5ex}
            \STATE Construct some $p_t \in \R^n $ \vspace{0.5ex}
            \STATE $x_{t+1} = \LMO_{\mathcal{X}}(p_t)$ \vspace{0.5ex}
            \STATE Construct some $v_1,\dots, v_t \in \R^n$ and a map $A_t$ \vspace{0.5ex}
            \STATE $y_{t+1} = A_t\left(\left\{\partial f_{y_i}(v_i), \partial (-f)_{v_i}(y_i),f(v_i, y_i)\right\}_{1 \leq i \leq t}\right)$ \vspace{0.2ex}
        \ENDFOR
        \vspace{0.3ex}
        \ENSURE $\bar x_T \in \conv\left(\{x_1,\dots,x_{T+1}\}\right),  \bar{y}_T = y_{T+1}$
    \end{algorithmic}
\end{algorithm}
\end{minipage}
\end{table*}

Given convex compact subsets $\cX \subseteq \R^n, \cY \subseteq \R^m$ with radius $R_{\cX} >0, R_{\cY} >0$, we denote by $\cF_{G_\x,G_\y, \|\cdot\|_2}^0(\cX,\cY)$ the class of all saddle-point problems~\cref{eq:minimax} whose payoff function $f$ is convex-concave and and for any $x,x' \in \cX, y,y' \in \cY$
\begin{align*}
    |f(x,y)-f(x',y)| &\leq G_\x\|x-x'\|_2 \quad \& \quad |f(x,y)-f(x,y')| \leq G_\y\|y-y'\|_2.
\end{align*}

\begin{theorem} \label{thm:lower-bound-lmo-lmo}
    Let $G_\x>0,G_\y>0, R_{\cX} >0, R_{\cY}>0$ be fixed. Then, for any generic LMO-LMO method of the form \cref{alg:generic-lmo-lmo} running for $T \geq 1$ iterations, there exist convex compact subsets $\cX \subseteq \R^n, \cY \subseteq \R^m$ with radii $R_{\cX}$ and $R_{\cY}$, respectively, and a function $f \in \cF_{G_\x,G_\y,\|\cdot\|_2}^0(\cX,\cY)$ with $T \leq  \min\{m,n\}/4-1$ such that
    \begin{align*}
        \max_{y\in\cY}f(\bar x_T,y)-\min_{x\in\cX}f(x,\bar y_T)
        \geq
        \frac{G_\x R_{\cX}}{2(T+1)^{1/2}}
        +
        \frac{G_\y R_{\cY}}{2(T+1)^{1/2}}.
    \end{align*}
\end{theorem}

\begin{theorem} \label{thm:lower-bound-lmo-po}
    Let $G_\x>0,G_\y>0, R_{\cX} >0, R_{\cY}>0$ be fixed. Then, for any generic LMO-PO method of the form \cref{alg:generic-lmo-po} running for $T \geq 1$ iterations, there exist convex compact subsets $\cX \subseteq \R^n, \cY \subseteq \R^m$ with radii $R_{\cX}$ and $R_{\cY}$, respectively, and a function $f \in \cF_{G_\x,G_\y,\|\cdot\|_2}^0(\cX,\cY)$ with $T \leq  \min\{m,n\}/4-1$ such that
    \begin{align*}
        \max_{y \in \cY} f(\bar{x}_T,y)-\min_{x \in \cX} f(x,\bar{y}_T)
        \geq
        \frac{G_\x R_{\cX}}{2(T+1)^{1/2}}
        +
        \frac{G_\y R_{\cY}}{2(T+1)^{1/2}}.
    \end{align*}
\end{theorem}

The proofs, presented in the appendix, closely follow the construction of \citet{nesterov2013introductory} and \citet{lan2013complexity}, adapted to the saddle-point setting.

\section{Numerical Results} \label{sec:numerical}
We evaluate the proposed methods on two saddle-point problems: matrix completion with spectral norm fit and robust multiclass classification with hinge loss. We compare the three LMO/PO variants in \cref{alg:lmo-lmo,alg:lmo-po,alg:po-lmo} with the fully projected baseline PO-PO, presented in \cref{alg:po-po} in Appendix~\ref{app:proofs}. Throughout the experiments, the suffix \texttt{Deterministic} means that the oracle returns the full subgradient of the payoff function, while the suffix \texttt{Stochastic} means that the oracle returns an unbiased stochastic estimate. Thus, for instance, \texttt{LMO-LMO (Deterministic)} refers to \cref{alg:lmo-lmo} implemented with the full oracle, whereas \texttt{LMO-LMO (Stochastic)} refers to the same method implemented with the stochastic oracle.
For LMO-based updates, we use simple enclosing sets $\myoverline{\cX} \supseteq \cX$ and $\myoverline{\cY} \supseteq \cY$ chosen so that projection onto them is inexpensive. The explicit projection formulas for these covering sets, together with all remaining implementation details and parameter choices, are provided in Appendix~\ref{app:numerics}. 
In the following, $\|\cdot\|_2$, $\|\cdot\|_{\mathrm{F}}$, $\|\cdot\|_\op$ and $\|\cdot\|_*$ denote the Euclidean norm, Frobenius norm, spectral (operator) norm and nuclear norm, respectively.

\paragraph{ Matrix Completion with Spectral Norm Fit.}
We first consider the bilinear saddle-point formulation of the spectral norm fit problem studied by \citet{juditsky2011solving}
\begin{equation} 
\label{eq:bilinear-MC}
    \min_{\|X\|_* \leq 1} ~ \max_{\|Y\|_* \leq 1} ~ \Tr\left[(\cA(X)-B)^\top Y\right],
\end{equation}
where $X\in\R^{n\times p}$, $Y\in\R^{m\times q}$, $B\in\R^{m\times q}$, and $\mathcal{A}:\R^{n\times p}\to\R^{m\times q}$ is a linear map. In this experiment, we have
\begin{align*}
    \cX = \{ X \in \R^{n \times p}: \|X\|_*\leq 1 \}
    \quad \text{and} \quad
    \cY = \{ Y \in \R^{m \times q}: \|Y\|_* \leq 1 \}.
\end{align*}
For the LMO-based variants, we choose the simple covering sets
\begin{align*}
    \myoverline{\cX} = \{X \in \R^{n\times p}: \|X\|_{\mathrm{F}} \leq 1\}
    \quad \text{and} \quad
    \myoverline{\cY} = \{Y \in \R^{m\times q}:\|Y\|_{\mathrm{F}} \leq 1\}.
\end{align*}
Projection onto these Frobenius balls is just a \emph{rescaling operation} and is therefore computationally cheap; see \cref{app:numerics} for details.
For this problem, we report the (strong) saddle gap, which admits a closed-form expression. Let $\mathcal{A}^*$ denote the adjoint of $\mathcal{A}$. Then, for any pair $(\bar X_T,\bar Y_T)$, the saddle gap
\begin{align*}
    &\max_{\|Y\|_*\leq 1} \Tr\left[(\cA(\bar X_T)-B)^\top Y \right]
    - \min_{\|X\|_*\leq 1} \Tr\left[(\cA(X)-B)^\top \bar Y_T \right] \\
    &= \|\mathcal{A}(\bar X_T)-B\|_{\op} + \|\mathcal{A}^*(\bar Y_T)\|_{\op}
    + \Tr\left[B^\top \bar Y_T \right].
\end{align*}
Since the payoff in \eqref{eq:bilinear-MC} is smooth, we also compare against the smoothing-based projection-free methods of \citet{giang2026projection}, the smoothing-based projection method of \citet{xu2023unified}, and the classical extra-gradient and optimistic gradient descent-ascent methods. These additional comparisons are reported in Appendix~\ref{app:numerics}.

\paragraph{Robust Multiclass Classification with Hinge Loss.}
We next consider a robust multiclass classification problem. Given a training set $\cD^{\mathrm{tr}}=\{(a_i,b_i)\}_{i=1}^n$, with features $a_i\in\R^d$ and labels $b_i\in\{1,\ldots,k\}$, the goal is to learn a predictor matrix $\Theta=[\theta_1,\ldots,\theta_k]^\top\in\R^{k\times d}$. The prediction for a new point $\hat a$ is $\hat b=\arg\max_{j\in[k]}\theta_j^\top \hat a$. We use the multiclass hinge loss $\ell_i(\Theta) = \max_{j\in[k]} \left\{ \one(j\neq b_i)+\theta_j^\top a_i-\theta_{b_i}^\top a_i \right\}.$
To promote low-rank structure, we impose the nuclear norm constraint $\cX=\{\Theta\in\R^{k\times d}:\|\Theta\|_*\leq r\}.$
Following the penalized distributionally robust optimization formulation in \citep{kuhn2025distributionally,ben2007old,gotoh2018robust}, we aim to solve
\begin{align}
\label{eq:numerics-dro}
    \min_{\Theta \in \cX} ~ \max_{y \in \cY} ~
    \frac{1}{n} \sum_{i=1}^{n} y_i \ell_i(\Theta) - \lambda \mathsf D_{\phi} \left( y, \tfrac{1}{n} \one_n \right),
\end{align}
where $\cY=\{y\in\R^n:y\geq 0,\sum_{i=1}^{n}y_i=1\}$ is the probability simplex. We use the Pearson $\chi^2$-divergence $\mathsf D_{\phi} (y,\frac{1}{n} \one_n) := \|ny - \one_n \|_2^2$, which leads to a form of variance regularization \citep{gotoh2018robust,lam2019recovering,duchi2019variance}. For the LMO-based update, we choose
\begin{align*}
    \myoverline{\cX} = \{\Theta \in \R^{k\times d} : \|\Theta\|_{\mathrm{F}} \leq r \}
    \quad \text{and} \quad \myoverline{\cY} = \{y \in \R^n : y \geq 0, \|y\|_2 \leq 1\}.
\end{align*}
Projection onto $\myoverline{\cX}$ is a Frobenius-norm rescaling, while projection onto $\myoverline{\cY}$ is obtained by nonnegative clipping followed by rescaling when necessary; see \cref{app:numerics} for more details.
For this experiment, we report the primal objective value as computing the full saddle gap is expensive at this scale, requiring solving both primal and dual problems. The primal objective is a tractable proxy and empirically reflects the same convergence behavior as the saddle-gap metric.

\paragraph{Summary of Results.}
Figure~\ref{fig:Experiment} reports the convergence of the LMO and PO variants in deterministic and stochastic settings. In the deterministic experiments, methods that avoid nuclear-norm projections perform best: \texttt{LMO-LMO (Deterministic)} achieves the fastest decrease and the best final value, with \texttt{LMO-PO (Deterministic)} typically close behind. This is consistent with the fact that an LMO over the nuclear-norm ball only requires a leading singular vector computation, whereas projection requires a substantially more expensive singular value decomposition. The gap is most visible in the classification experiment, where the primal nuclear-norm projection is the main computational bottleneck.
The stochastic experiments show a similar pattern, with additional variability due to sampling. \texttt{LMO-LMO (Stochastic)} performs best on matrix completion, while \texttt{LMO-PO (Stochastic)} is especially competitive on classification. For the classification task, LMO and PO updates have comparable computational cost over $\cY$. When dual projections are feasible, they can therefore provide additional stability without introducing a significant computational overhead, which helps explain the strong performance of \texttt{LMO-PO (Stochastic)} in terms of the primal objective. By contrast, using a projection on the primal nuclear-norm constraint remains costly and leads to slower progress in wall-clock time. Overall, the results support the main message of the paper: replacing costly projections with LMOs can yield clear computational gains without sacrificing the optimal stochastic rate.
\vspace{-1em}
\begin{figure*}[!b]
    \centering
    \begin{subfigure}[b]{0.495\textwidth}
        \centering
        \includegraphics[width =1\linewidth]{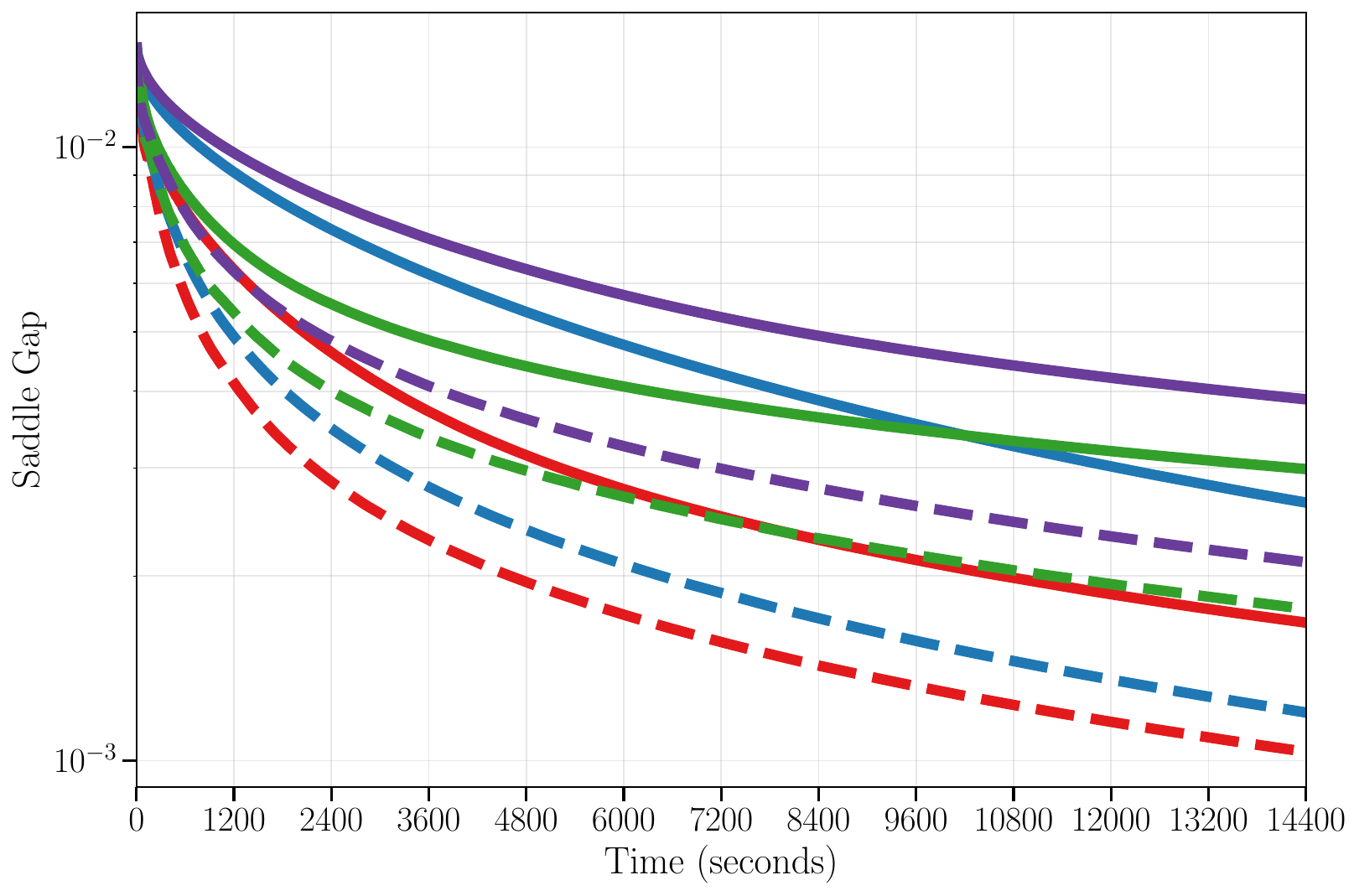}
        \label{subfig:matrix} \vspace{-1em}
    \end{subfigure}
    \hfill
    \begin{subfigure}[b]{0.495\textwidth}
        \centering
        \includegraphics[ width =1\linewidth]{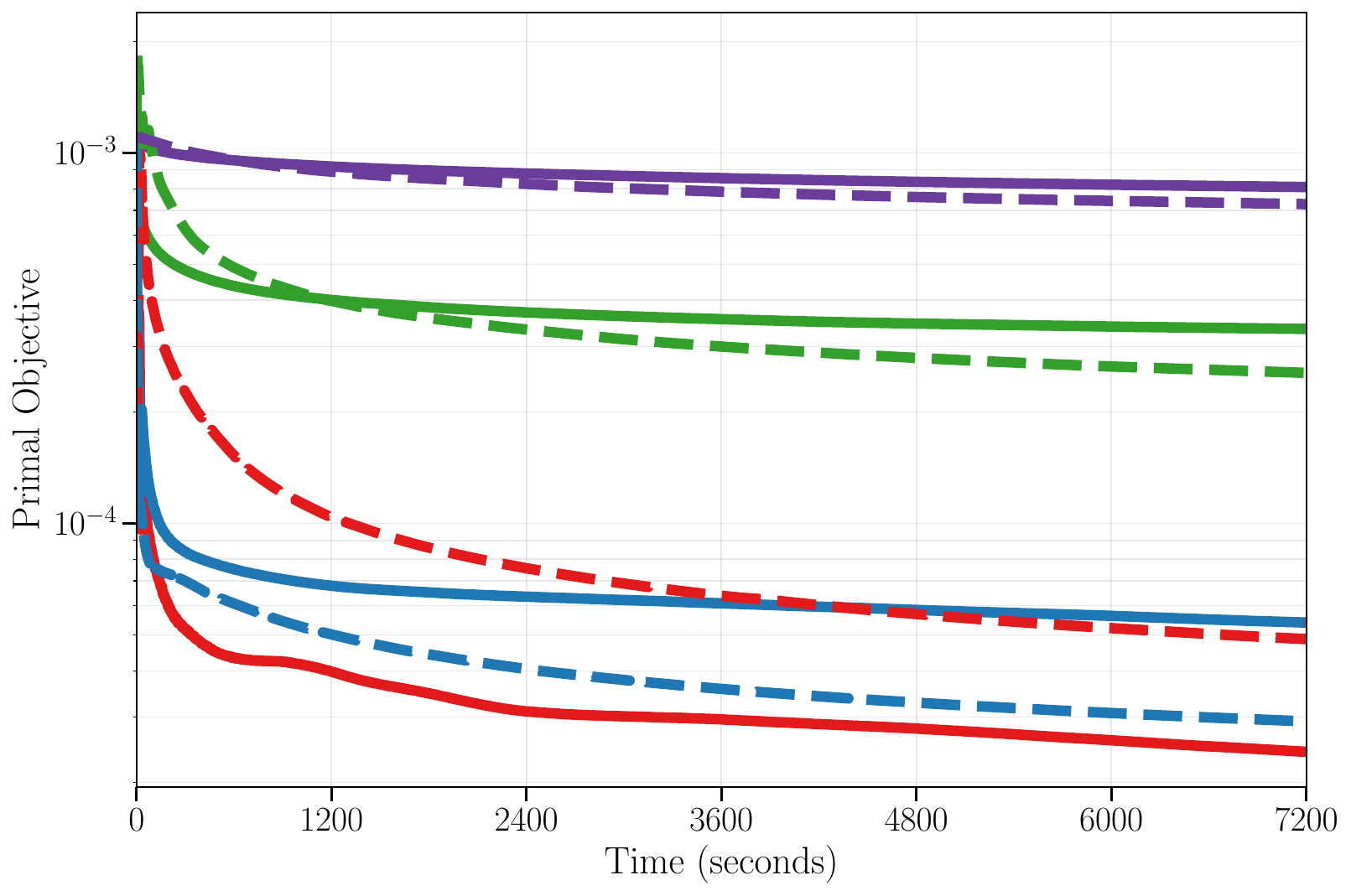}
        \label{subfig:hinge} \vspace{-1em}
    \end{subfigure} 
    \begin{subfigure}[b]{0.95\textwidth}
    \includegraphics[width=\linewidth]{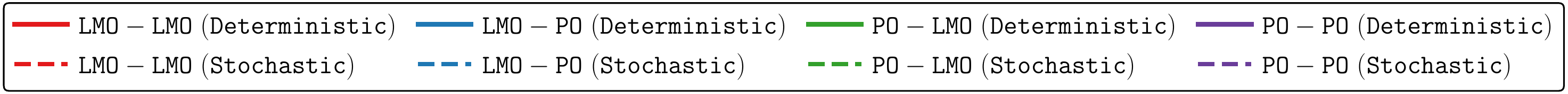}
    \end{subfigure}
    \caption{Convergence comparison of deterministic and stochastic LMO/PO variants.}
    \label{fig:Experiment}
\end{figure*}

\bibliographystyle{myabbrvnat}
\bibliography{bib} 

@inproceedings{boroun2023projection,
  title={Projection-free methods for solving nonconvex-concave saddle point problems},
  author={Boroun, Morteza and Yazdandoost Hamedani, Erfan and Jalilzadeh, Afrooz},
  booktitle={Advances in Neural Information Processing Systems},
  pages={53844--53856},
  year={2023}
}

@book{bertsekas1999nonlinear,
  title={Nonlinear Programming},
  author={Bertsekas, Dimitri},
  year={1999},
  publisher={Athena Scientific}
}

@article{nesterov2005smooth,
  title={Smooth minimization of non-smooth functions},
  author={Nesterov, Yu},
  journal={Mathematical Programming},
  volume={103},
  number={1},
  pages={127--152},
  year={2005},
  publisher={Springer}
}

@article{nemirovski2004prox,
  title={{Prox-method with rate of convergence $O(1/t)$ for variational inequalities with Lipschitz continuous monotone operators and smooth convex-concave saddle point problems}},
  author={Nemirovski, Arkadi},
  journal={SIAM Journal on Optimization},
  volume={15},
  number={1},
  pages={229--251},
  year={2004},
  publisher={SIAM}
}

@article{xu2023unified,
  title={A unified single-loop alternating gradient projection algorithm for nonconvex--concave and convex--nonconcave minimax problems},
  author={Xu, Zi and Zhang, Huiling and Xu, Yang and Lan, Guanghui},
  journal={Mathematical Programming},
  volume={201},
  number={1},
  pages={635--706},
  year={2023},
  publisher={Springer}
}

@book{nesterov2013introductory,
  title={Introductory Lectures on Convex Optimization: A Basic Course},
  author={Nesterov, Yurii},
  year={2013},
  publisher={Springer}
}

@article{nesterov1983method,
  title={{A method for solving the convex programming problem with convergence rate $O(1/k^2)$}},
  author={Nesterov, Yurii},
  journal={Dokl akad nauk Sssr},
  volume={27},
  number={2},
  pages={372--376},
  year={1983}
}

@inproceedings{jaggi2013revisiting,
  title={Revisiting {F}rank-{W}olfe: {P}rojection-free sparse convex optimization},
  author={Jaggi, Martin},
  booktitle={International Conference on Machine Learning},
  pages={427--435},
  year={2013}
}

@article{korpelevich1976extragradient,
  title={The extragradient method for finding saddle points and other problems},
  author={Korpelevich, Galina M},
  journal={Matecon},
  volume={12},
  pages={747--756},
  year={1976}
}

@article{malitsky2015projected,
  title={Projected reflected gradient methods for monotone variational inequalities},
  author={Malitsky, Yu},
  journal={SIAM Journal on Optimization},
  volume={25},
  number={1},
  pages={502--520},
  year={2015},
  publisher={SIAM}
}

@article{tseng2000modified,
  title={A modified forward-backward splitting method for maximal monotone mappings},
  author={Tseng, Paul},
  journal={SIAM Journal on Control and Optimization},
  volume={38},
  number={2},
  pages={431--446},
  year={2000},
  publisher={SIAM}
}

@article{nesterov2007dual,
  title={Dual extrapolation and its applications to solving variational inequalities and related problems},
  author={Nesterov, Yurii},
  journal={Mathematical Programming},
  volume={109},
  number={2},
  pages={319--344},
  year={2007},
  publisher={Springer}
}

@article{nedic2009subgradient,
  title={Subgradient methods for saddle-point problems},
  author={Nedi{\'c}, Angelia and Ozdaglar, Asuman},
  journal={Journal of Optimization Theory and Applications},
  volume={142},
  number={1},
  pages={205--228},
  year={2009},
  publisher={Springer}
}

@article{chambolle2011first,
  title={A first-order primal-dual algorithm for convex problems with applications to imaging},
  author={Chambolle, Antonin and Pock, Thomas},
  journal={Journal of Mathematical Imaging and Vision},
  volume={40},
  number={1},
  pages={120--145},
  year={2011},
  publisher={Springer}
}

@book{braun2022conditional,
  title={Conditional Gradient Methods: From Core Principles to AI Applications},
  author={Braun, G{\'a}bor and Carderera, Alejandro and Combettes, Cyrille W and Hassani, Hamed and Karbasi, Amin and Mokhtari, Aryan and Pokutta, Sebastian},
  publisher={SIAM},
  year={2025}
}

@article{guzman2015lower,
  title={On lower complexity bounds for large-scale smooth convex optimization},
  author={Guzm{\'a}n, Crist{\'o}bal and Nemirovski, Arkadi},
  journal={Journal of Complexity},
  volume={31},
  number={1},
  pages={1--14},
  year={2015},
  publisher={Elsevier}
}

@article{ouyang2021lower,
  title={Lower complexity bounds of first-order methods for convex-concave bilinear saddle-point problems},
  author={Ouyang, Yuyuan and Xu, Yangyang},
  journal={Mathematical Programming},
  volume={185},
  number={1},
  pages={1--35},
  year={2021},
  publisher={Springer}
}

@inproceedings{nouiehed2019solving,
  title={Solving a class of non-convex min-max games using iterative first order methods},
  author={Nouiehed, Maher and Sanjabi, Maziar and Huang, Tianjian and Lee, Jason D and Razaviyayn, Meisam},
  booktitle={Advances in Neural Information Processing Systems},
  pages={14934--14942},
  year={2019}
}

@inproceedings{gidel2017frank,
  title={{Frank-Wolfe} algorithms for saddle point problems},
  author={Gidel, Gauthier and Jebara, Tony and Lacoste-Julien, Simon},
  booktitle={International Conference on Artificial Intelligence and Statistics},
  pages={362--371},
  year={2017}
}

@inproceedings{chen2020efficient,
  title={Efficient projection-free algorithms for saddle point problems},
  author={Chen, Cheng and Luo, Luo and Zhang, Weinan and Yu, Yong},
  booktitle={Advances in Neural Information Processing Systems},
  pages={10799--10808},
  year={2020}
}

@inproceedings{chen2024last,
  title={Last-iterate convergence for generalized {F}rank-{W}olfe in monotone variational inequalities},
  author={Chen, Zaiwei and Mazumdar, Eric},
  booktitle={Advances in Neural Information Processing Systems},
  pages={115440--115467},
  year={2024}
}

@inproceedings{he2015semi,
  title={Semi-proximal mirror-prox for nonsmooth composite minimization},
  author={He, Niao and Harchaoui, Zaid},
  booktitle={Advances in Neural Information Processing Systems},
  pages={3411--3419},
  year={2015}
}

@inproceedings{suggala2020follow,
  title={Follow the perturbed leader: {O}ptimism and fast parallel algorithms for smooth minimax games},
  author={Suggala, Arun and Netrapalli, Praneeth},
  booktitle={Advances in Neural Information Processing Systems},
  pages={22316--22326},
  year={2020}
}

@inproceedings{kolmogorov2021one,
  title={One-sided {F}rank-{W}olfe algorithms for saddle problems},
  author={Kolmogorov, Vladimir and Pock, Thomas},
  booktitle={International Conference on Machine Learning},
  pages={5665--5675},
  year={2021}
}

@article{popov1980modification,
  title={A modification of the {Arrow-Hurwicz} method for search of saddle points},
  author={Popov, Leonid Denisovich},
  journal={Mathematical Notes of the Academy of Sciences of the USSR},
  volume={28},
  number={5},
  pages={845--848},
  year={1980},
  publisher={Springer}
}

@phdthesis{hammond1984solving,
  title={Solving Asymmetric Variational Inequality Problems and Systems of Equations with Generalized Nonlinear Programming Algorithms},
  author={Hammond, Janice H},
  year={1984},
  school={Massachusetts Institute of Technology}
}

@article{lan2013complexity,
  title={The complexity of large-scale convex programming under a linear optimization oracle},
  author={Lan, Guanghui},
  journal={arXiv:1309.5550},
  year={2013}
}

@article{lan2016conditional,
  title={Conditional gradient sliding for convex optimization},
  author={Lan, Guanghui and Zhou, Yi},
  journal={SIAM Journal on Optimization},
  volume={26},
  number={2},
  pages={1379--1409},
  year={2016},
  publisher={SIAM}
}

@article{lan2016gradient,
  title={Gradient sliding for composite optimization},
  author={Lan, Guanghui},
  journal={Mathematical Programming},
  volume={159},
  number={1},
  pages={201--235},
  year={2016},
  publisher={Springer}
}

@article{baghbadorani2025frank,
  title={A {F}rank-{W}olfe Algorithm for Strongly Monotone Variational Inequalities},
  author={Baghbadorani, Reza Rahimi and Esfahani, Peyman Mohajerin and Grammatico, Sergio},
  journal={Operations Research Letters},
  pages={107388},
  year={2025},
  publisher={Elsevier}
}

@inproceedings{yurtsever2019conditional,
  title={A conditional-gradient-based augmented {L}agrangian framework},
  author={Yurtsever, Alp and Fercoq, Olivier and Cevher, Volkan},
  booktitle={International Conference on Machine Learning},
  pages={7272--7281},
  year={2019}
}

@article{asgari2024nonsmooth,
  title={Nonsmooth projection-free optimization with functional constraints},
  author={Asgari, Kamiar and Neely, Michael J},
  journal={Computational Optimization and Applications},
  volume={89},
  number={3},
  pages={927--975},
  year={2024},
  publisher={Springer}
}

@article{asgari2022projection,
  title={Projection-free non-smooth convex programming},
  author={Asgari, Kamiar and Neely, Michael J},
  journal={arXiv:2208.05127},
  year={2022}
}

@article{lan2021conditional,
    title={Conditional gradient methods for convex optimization with general affine and nonlinear constraints},
    author={Lan, Guanghui and Romeijn, Edwin and Zhou, Zhiqiang},
    journal={SIAM Journal on Optimization},
    volume={31},
    number={3},
    pages={2307--2339},
    year={2021},
    publisher={SIAM}
}

@article{duchi2019variance,
  title={Variance-based regularization with convex objectives},
  author={Duchi, John and Namkoong, Hongseok},
  journal={Journal of Machine Learning Research},
  volume={20},
  number={68},
  pages={1--55},
  year={2019}
}

@article{kuhn2025distributionally,
  title={Distributionally robust optimization},
  author={Kuhn, Daniel and Shafiee, Soroosh and Wiesemann, Wolfram},
  journal={Acta Numerica},
  volume={34},
  pages={579--804},
  year={2025},
  publisher={Cambridge University Press}
}

@article{ben2007old,
  title={An old-new concept of convex risk measures: {T}he optimized certainty equivalent},
  author={Ben-Tal, Aharon and Teboulle, Marc},
  journal={Mathematical Finance},
  volume={17},
  number={3},
  pages={449--476},
  year={2007}
}

@article{lam2019recovering,
  title={Recovering best statistical guarantees via the empirical divergence-based distributionally robust optimization},
  author={Lam, Henry},
  journal={Operations Research},
  volume={67},
  number={4},
  pages={1090--1105},
  year={2019}
}

@article{gotoh2018robust,
  title={Robust empirical optimization is almost the same as mean--variance optimization},
  author={Gotoh, Jun-ya and Kim, Michael Jong and Lim, Andrew EB},
  journal={Operations Research Letters},
  volume={46},
  number={4},
  pages={448--452},
  year={2018},
  publisher={Elsevier}
}

@article{nesterov2009primal,
  title={Primal-dual subgradient methods for convex problems},
  author={Nesterov, Yurii},
  journal={Mathematical Programming},
  volume={120},
  number={1},
  pages={221--259},
  year={2009},
  publisher={Springer}
}

@book{shor1985minimization,
  title={Minimization Methods for Non-Differentiable Functions},
  author={Shor, Naum Zuselevich},
  year={1985},
  publisher={Springer}
}

@article{robbins1951stochastic,
  title={A stochastic approximation method},
  author={Robbins, Herbert and Monro, Sutton},
  journal={The Annals of Mathematical Statistics},
  pages={400--407},
  year={1951},
  publisher={JSTOR}
}

@article{nemirovski2009robust,
  title={Robust stochastic approximation approach to stochastic programming},
  author={Nemirovski, Arkadi and Juditsky, Anatoli and Lan, Guanghui and Shapiro, Alexander},
  journal={SIAM Journal on Optimization},
  volume={19},
  number={4},
  pages={1574--1609},
  year={2009},
  publisher={SIAM}
}

@article{giang2026projection,
  title={Projection-Free Algorithms for Minimax Problems},
  author={Giang-Tran, Khanh-Hung and Shafiee, Soroosh and Ho-Nguyen, Nam},
  journal={arXiv:2603.29870},
  year={2026}
}

@inproceedings{thekumparampil2020projection,
  title={Projection efficient subgradient method and optimal nonsmooth {F}rank-{W}olfe method},
  author={Thekumparampil, Kiran K and Jain, Prateek and Netrapalli, Praneeth and Oh, Sewoong},
  booktitle={Advances in Neural Information Processing Systems},
  pages={12211--12224},
  year={2020}
}

@inproceedings{thekumparampil2020optimal,
  title={Optimal nonsmooth {F}rank-{W}olfe method for stochastic regret minimization},
  author={Thekumparampil, Kiran Koshy and Jain, Prateek and Netrapalli, Praneeth and Oh, Sewoong},
  booktitle={Optimization for Machine Learning Workshop},
  year={2020}
}

@inproceedings{hazan2012projection,
  title={Projection-free online learning},
  author={Hazan, Elad and Kale, Satyen},
  booktitle={International Conference on Machine Learning},
  pages={1843--1850},
  year={2012}
}

@inproceedings{hazan2020faster,
  title={Faster projection-free online learning},
  author={Hazan, Elad and Minasyan, Edgar},
  booktitle={Conference on Learning Theory},
  pages={1877--1893},
  year={2020}
}

@article{white1993extension,
  title={Extension of the {F}rank-{W}olfe algorithm to concave nondifferentiable objective functions},
  author={White, Douglas J},
  journal={Journal of Optimization Theory and Applications},
  volume={78},
  number={2},
  pages={283--301},
  year={1993},
  publisher={Springer}
}

@article{ravi2019deterministic,
  title={A deterministic nonsmooth {F}rank {W}olfe algorithm with coreset guarantees},
  author={Ravi, Sathya N and Collins, Maxwell D and Singh, Vikas},
  journal={INFORMS Journal on Optimization},
  volume={1},
  number={2},
  pages={120--142},
  year={2019},
  publisher={INFORMS}
}

@article{cheung2017nonsmooth,
  title={Nonsmooth {F}rank-{W}olfe using uniform affine approximations},
  author={Cheung, Edward and Li, Yuying},
  journal={arXiv:1710.05776},
  year={2017}
}

@inproceedings{lu2023projection,
  title={Projection-free adaptive regret with membership oracles},
  author={Lu, Zhou and Brukhim, Nataly and Gradu, Paula and Hazan, Elad},
  booktitle={Conference on Algorithmic Learning Theory},
  pages={1055--1073},
  year={2023}
}

@article{grimmer2024radialI,
  title={Radial duality part {I}: {Foundations}},
  author={Grimmer, Benjamin},
  journal={Mathematical Programming},
  volume={205},
  number={1},
  pages={33--68},
  year={2024},
  publisher={Springer}
}

@article{grimmer2024radialII,
  title={Radial duality part {II}: {A}pplications and algorithms},
  author={Grimmer, Benjamin},
  journal={Mathematical Programming},
  volume={205},
  number={1},
  pages={69--105},
  year={2024},
  publisher={Springer}
}

@article{tao2019strength,
  title={The strength of {N}esterov's extrapolation in the individual convergence of nonsmooth optimization},
  author={Tao, Wei and Pan, Zhisong and Wu, Gaowei and Tao, Qing},
  journal={IEEE Transactions on Neural Networks and Learning Systems},
  volume={31},
  number={7},
  pages={2557--2568},
  year={2019},
  publisher={IEEE}
}

@book{nemirovski1983problem,
  title={Problem Complexity and Method Efficiency in Optimization},
  author={Nemirovsky, Arkadi Semenovi{\v{c}} and Yudin, David Borisovich},
  year={1983},
  publisher={Wiley}
}

@article{hough2026asymptotic,
  title={Asymptotic Convergence of the {F}rank-{W}olfe Algorithm for Monotone Variational Inequalities},
  author={Hough, Matthew},
  journal={arXiv:2603.12104},
  year={2026}
}

@article{juditsky2016solving,
  title={Solving variational inequalities with monotone operators on domains given by linear minimization oracles},
  author={Juditsky, Anatoli and Nemirovski, Arkadi},
  journal={Mathematical Programming},
  volume={156},
  number={1},
  pages={221--256},
  year={2016},
  publisher={Springer}
}

@article{cox2014dual,
  title={Dual subgradient algorithms for large-scale nonsmooth learning problems},
  author={Cox, Bruce and Juditsky, Anatoli and Nemirovski, Arkadi},
  journal={Mathematical Programming},
  volume={148},
  number={1},
  pages={143--180},
  year={2014},
  publisher={Springer}
}

@article{harchaoui2015conditional,
  title={Conditional gradient algorithms for norm-regularized smooth convex optimization},
  author={Harchaoui, Zaid and Juditsky, Anatoli and Nemirovski, Arkadi},
  journal={Mathematical Programming},
  volume={152},
  number={1},
  pages={75--112},
  year={2015},
  publisher={Springer}
}

@inproceedings{locatello2019stochastic,
  title={Stochastic {F}rank-{W}olfe for composite convex minimization},
  author={Locatello, Francesco and Yurtsever, Alp and Fercoq, Olivier and Cevher, Volkan},
  booktitle={Advances in Neural Information Processing Systems},
  pages={14291--14301},
  year={2019}
}

@article{alacaoglu2023beyond,
  title={Beyond the golden ratio for variational inequality algorithms},
  author={Alacaoglu, Ahmet and B{\"o}hm, Axel and Malitsky, Yura},
  journal={Journal of Machine Learning Research},
  volume={24},
  number={172},
  pages={1--33},
  year={2023}
}

@article{malitsky2020golden,
  title={Golden ratio algorithms for variational inequalities},
  author={Malitsky, Yura},
  journal={Mathematical Programming},
  volume={184},
  number={1},
  pages={383--410},
  year={2020},
  publisher={Springer}
}

@inproceedings{daskalakis2018training,
  title={Training {GANs} with Optimism},
  author={Daskalakis, Constantinos and Ilyas, Andrew and Syrgkanis, Vasilis and Zeng, Haoyang},
  booktitle={International Conference on Learning Representations},
  year={2018}
}

@article{malitsky2020forward,
  title={A forward-backward splitting method for monotone inclusions without cocoercivity},
  author={Malitsky, Yura and Tam, Matthew K},
  journal={SIAM Journal on Optimization},
  volume={30},
  number={2},
  pages={1451--1472},
  year={2020},
  publisher={SIAM}
}

@article{juditsky2011solving,
  title={Solving variational inequalities with stochastic mirror-prox algorithm},
  author={Juditsky, Anatoli and Nemirovski, Arkadi and Tauvel, Claire},
  journal={Stochastic Systems},
  volume={1},
  number={1},
  pages={17--58},
  year={2011},
  publisher={INFORMS}
}

@inproceedings{agarwal2009information,
  title={Information-theoretic lower bounds on the oracle complexity of convex optimization},
  author={Agarwal, Alekh and Bartlett, Peter and Ravikumar, Pradeep and Wainwright, Martin J},
  booktitle={Advances in Neural Information Processing Systems},
  pages={1--9},
  year={2009}
}

@inproceedings{foster2019complexity,
  title={The complexity of making the gradient small in stochastic convex optimization},
  author={Foster, Dylan J and Sekhari, Ayush and Shamir, Ohad and Srebro, Nathan and Sridharan, Karthik and Woodworth, Blake},
  booktitle={Conference on Learning Theory},
  pages={1319--1345},
  year={2019},
  organization={PMLR}
}

@article{grimmer2026lower,
  title={Lower Bounds for Linear Minimization Oracle Methods Optimizing over Strongly Convex Sets},
  author={Grimmer, Benjamin and Liu, Ning},
  journal={arXiv:2602.22608},
  year={2026}
}

@article{halbey2026lower,
  title={Lower Bounds for {F}rank-{W}olfe on Strongly Convex Sets},
  author={Halbey, Jannis and Deza, Daniel and Zimmer, Max and Roux, Christophe and Stellato, Bartolomeo and Pokutta, Sebastian},
  journal={arXiv:2602.04378},
  year={2026}
}

\appendix
\numberwithin{equation}{section}
\makeatletter
\@addtoreset{proposition}{section}
\def\theproposition{\thesection.\arabic{proposition}}
\@addtoreset{theorem}{section}
\def\thetheorem{\thesection.\arabic{theorem}}
\@addtoreset{lemma}{section}
\def\thelemma{\thesection.\arabic{lemma}}
\makeatother

\section{Proofs}
\label{app:proofs}
\subsection{Proofs of \texorpdfstring{Proposition~\protect\ref{prop:generic-lmo}}{Proposition~1}}

For the convenience, we repeat the statement of \cref{prop:generic-lmo} below.

\begin{proposition}[Restatement of Proposition~\ref{prop:generic-lmo}]
\label{prop:generic-lmo-restate}
    Suppose Assumptions~\ref{ass:extended} and \ref{ass:stochastic-oracle} hold.
    \begin{enumerate}[label=(\roman*),leftmargin=2em]
        \item \label{prop:generic-lmo:primal}
        Let $\{v_t,x_t,\lambda_t\}_{t \geq 1}$ be the primal sequences generated by Algorithms~\ref{alg:lmo-lmo} or \ref{alg:lmo-po} with parameters satisfying $\alpha_t \leq \alpha_{t+1}$ and $\eta_t \geq \eta_{t+1}$ for all $t \geq 1$. 
        Define the dual evaluation sequence $\{z_t\}_{t \geq 1} = \{ u_t \}_{t \geq 1}$ for \cref{alg:lmo-lmo} and $\{z_t\}_{t \geq 1} = \{ y_t \}_{t \geq 1}$ for \cref{alg:lmo-po}.
        Then, for every $T \geq 2$, we have 
        \begin{equation}
        \label{eq:generic-primal-lmo-2}
        \begin{aligned}
            &\E\left[\frac{\eta_{T-1}}{2} \left\|\frac{\lambda_T}{\eta_T}\right\|_2^2 + \max_{x \in \cX} \; \sum_{t = 1}^{T} \big( f(v_t,z_t) - f(x,z_t) \big) \right]  \\
            &\leq 2\left(\alpha_{T-1} +\sum_{t = 1}^{T-1}\eta_t\right) R_{\cX}^2+\frac{G_\x^2}{2\alpha_{T-1}}+\sum_{t = 1}^{T-1}\frac{G_\x^2}{2 \alpha_t}+ 2G_\x R_{\cX} T^{1/2}. 
        \end{aligned}  
        \end{equation}
    
        \item \label{prop:generic-lmo:dual}
        Let $\{u_t,y_t,\mu_t\}_{t \geq 1}$ be the dual sequences generated by Algorithms~\ref{alg:lmo-lmo} or \ref{alg:po-lmo} with parameters satisfying $\beta_t \leq \beta_{t+1}$ and $\tau_t \geq \tau_{t+1}$ for all $t \geq 1$.
        Define the primal evaluation sequence $\{w_t\}_{t \geq 1} = \{ v_t \}_{t \geq 1}$ for \cref{alg:lmo-lmo} and $\{w_t\}_{t \geq 1} = \{ x_t \}_{t \geq 1}$ for \cref{alg:po-lmo}.
        Then, for every $T \geq 2$, we have 
        \begin{equation}
        \label{eq:generic-dual-lmo-2}
        \begin{aligned}
            &\E\left[\frac{\tau_{T-1}}{2} \left\|\frac{\mu_T}{\tau_T}\right\|_2^2 + \max_{y \in \cY} \sum_{t = 1}^{T}  \big( f(w_t, y) - f(w_t, u_t) \big) \right]\\
            &\leq 2\left(\beta_{T-1} +\sum_{t = 1}^{T-1}\tau_t\right) R_{\cY}^2+\frac{G_\y^2}{2\beta_{T-1}}+\sum_{t = 1}^{T-1}\frac{G_\y^2}{2 \beta_t}+ 2G_\y R_{\cY}T^{1/2}.
        \end{aligned}
        \end{equation}
    \end{enumerate}
\end{proposition}

The proof uses the following standard Euclidean three-point inequality; see also \citep[Lemma~6]{nesterov2009primal}.

\begin{lemma} \label{lemma:triangle}
    Let $h$ be a convex function defined on a nonempty convex set $\cC \subseteq \R^d$. Given $\alpha > 0$ and $\bar x \in \cC$, let
    $$x' = \argmin_{x \in \cC} \left\{ h(x) + \alpha \| x - \bar{x} \|_2^2 \right\}.$$
    Then, for any $x \in \cC$,
    $$h(x')+\alpha \|x'-\bar{x}\|_2^2 \leq h(x)+ \alpha \|x-\bar{x}\|_2^2 - \alpha \|x-x'\|_2^2.$$
\end{lemma}

\begin{proof}[Proof of \cref{lemma:triangle}]
    We provide a proof for completeness. Define the function 
    $$\phi(x) := h(x) + \alpha \|x-\bar x\|_2^2, \quad x \in \cC.$$
    By $2 \alpha$-strong convexity of $\phi$ on $\cC$, for every $x \in \cC$ and every $\theta \in [0,1]$, we have
    $$ \phi \big((1-\theta)x' + \theta x\big) \leq (1-\theta) \phi(x') + \theta \phi(x) - \alpha \theta(1-\theta) \|x-x'\|_2^2. $$
    Since $x'$ minimizes $\phi$ over $\cC$ and $(1-\theta)x' + \theta x \in \cC$, we also have 
    $$\phi(x') \leq \phi \big((1-\theta)x' + \theta x \big).$$    
    Combining the last two displays yields
    $$ \phi(x') \leq (1-\theta) \phi(x') + \theta \phi(x) - \alpha \theta (1-\theta) \|x-x'\|_2^2, $$
    and therefore, for every $\theta \in (0,1]$, $\phi(x') \leq \phi(x) - \alpha(1-\theta) \|x-x'\|_2^2$. Letting $\theta \downarrow 0$ gives
    $$ \phi(x') \leq \phi(x) - \alpha\|x-x'\|_2^2. $$
    Expanding the definition of $\phi$ completes the proof.    
\end{proof}

We are now ready to present the proof.

\begin{proof}[Proof of \cref{prop:generic-lmo-restate}]
    By symmetry, it suffices to prove part~\ref{prop:generic-lmo:primal}. 
    The proof of part~\ref{prop:generic-lmo:dual} follows by exchanging
    \begin{align*}
        (v_t,x_t,\lambda_t,\alpha_t,\eta_t,\cX,R_{\cX},G_\x,z_t)
        \quad \text{with} \quad
        (u_t,y_t,\mu_t,\beta_t,\tau_t,\cY,R_{\cY},G_\y,w_t),
    \end{align*}
    and replacing $f_{z_t}(\cdot)$ by $(-f)_{w_t}(\cdot)$. 
    We therefore focus on the primal inequality.

    Fix any $x \in \cX$. The update of $v_{t+1}$ can be written as
    \begin{align*}
        v_{t+1}
        = 
        \argmin_{v \in \myoverline{\cY}} \left\{(\lambda_t+g^\x_t)^\top v+\frac{\alpha_t}{2}\|v-v_t\|_2^2\right\}.
    \end{align*}
    Since $x \in \cX \subseteq \myoverline{\cX}$, invoking \cref{lemma:triangle} with $h(v):=(\lambda_t+g^\x_t)^\top v$, $\cC:=\myoverline{\cY}$, $\bar x:=v_t$ and $\alpha:=\alpha_t/2$ gives
    \begin{align*}
        (\lambda_t+g^\x_t)^\top v_{t+1}+\frac{\alpha_t}{2}\|v_{t+1}-v_t\|_2^2 
        \leq (\lambda_t+g^\x_t)^\top x+\frac{\alpha_t}{2}\|x-v_t\|_2^2-\frac{\alpha_t}{2}\|x-v_{t+1}\|_2^2.
    \end{align*}
    Subtracting $\lambda_t^\top x_{t+1}$ and $(g^\x_t)^\top v_t$ from both sides yields
    \begin{align*}
        &\lambda_t^\top(v_{t+1}-x_{t+1})+(g^\x_t)^\top(v_{t+1}-v_t)+\frac{\alpha_t}{2}\|v_{t+1}-v_t\|_2^2 \\
        &\leq \lambda_t^\top(x-x_{t+1})+(g^\x_t)^\top(x-v_t)+\frac{\alpha_t}{2}\|x-v_t\|_2^2-\frac{\alpha_t}{2}\|x-v_{t+1}\|_2^2.
    \end{align*}
    Since $x_{t+1}=\LMO_{\cX}(-\lambda_t)$, we have $\lambda_t^\top(x-x_{t+1})\leq 0$. Hence,
    \begin{align*}
        &\lambda_t^\top(v_{t+1}-x_{t+1})+(g^\x_t)^\top(v_{t+1}-v_t)+\frac{\alpha_t}{2}\|v_{t+1}-v_t\|_2^2 \\
        &\leq (g^\x_t)^\top(x-v_t)+\frac{\alpha_t}{2}\|x-v_t\|_2^2-\frac{\alpha_t}{2}\|x-v_{t+1}\|_2^2.
    \end{align*}
    The update of $\lambda_{t+1}$ implies ${\lambda_{t+1}}/{\eta_{t+1}} = {\lambda_t}/{\eta_t}+v_{t+1}-x_{t+1}$.
    Therefore,
    \begin{align*}
        \lambda_t^\top(v_{t+1}-x_{t+1})
        &=\eta_t \left( \frac{1}{\eta_t} \lambda_t^\top (v_{t+1} - x_{t+1}) \right)
        =\frac{\eta_t}{2}\left(\left\|\frac{\lambda_{t+1}}{\eta_{t+1}}\right\|_2^2-\left\|\frac{\lambda_t}{\eta_t}\right\|_2^2-\|v_{t+1}-x_{t+1}\|_2^2\right).
    \end{align*}
    Substituting this identity into the previous inequality gives
    \begin{align*}
        &\frac{\eta_t}{2}\left(\left\|\frac{\lambda_{t+1}}{\eta_{t+1}}\right\|_2^2-\left\|\frac{\lambda_t}{\eta_t}\right\|_2^2-\|v_{t+1}-x_{t+1}\|_2^2\right) 
        +(g^\x_t)^\top(v_{t+1}-v_t)+\frac{\alpha_t}{2}\|v_{t+1}-v_t\|_2^2 \\
        &\leq (g^\x_t)^\top(x-v_t)+\frac{\alpha_t}{2}\|x-v_t\|_2^2-\frac{\alpha_t}{2}\|x-v_{t+1}\|_2^2.
    \end{align*}
    By Young's inequality,
    \begin{align*}
        (g^\x_t)^\top(v_{t+1}-v_t)+\frac{\alpha_t}{2}\|v_{t+1}-v_t\|_2^2 \geq -\frac{\|g^\x_t\|_2^2}{2\alpha_t}.
    \end{align*}
    Also, since $v_{t+1}\in\myoverline{\cX}$ and $x_{t+1}\in\cX$, we use the crude bound $\|v_{t+1}-x_{t+1}\|_2\leq 2R_{\cX}$. Thus,
    \begin{align*}
        &\frac{\eta_t}{2}\left(\left\|\frac{\lambda_{t+1}}{\eta_{t+1}}\right\|_2^2-\left\|\frac{\lambda_t}{\eta_t}\right\|_2^2\right)-\frac{\|g^\x_t\|_2^2}{2\alpha_t} 
        \leq (g^\x_t)^\top(x-v_t)+\frac{\alpha_t}{2}\|x-v_t\|_2^2-\frac{\alpha_t}{2}\|x-v_{t+1}\|_2^2+2\eta_tR_{\cX}^2.
    \end{align*}
    Summing from $t=1$ to $T-1$ gives
    \begin{align*}
        &\frac{\eta_{T-1}}{2}\left\|\frac{\lambda_T}{\eta_T}\right\|_2^2+\frac{1}{2}\sum_{t=2}^{T-1}(\eta_{t-1}-\eta_t)\left\|\frac{\lambda_t}{\eta_t}\right\|_2^2-\frac{\eta_1}{2}\left\|\frac{\lambda_1}{\eta_1}\right\|_2^2-\sum_{t=1}^{T-1}\frac{\|g^\x_t\|_2^2}{2\alpha_t} \\
        &\leq \sum_{t=1}^{T-1}(g^\x_t)^\top(x-v_t)+\frac{\alpha_1}{2}\|x-v_1\|_2^2+\frac{1}{2}\sum_{t=2}^{T-1}(\alpha_t-\alpha_{t-1})\|x-v_t\|_2^2-\frac{\alpha_{T-1}}{2}\|x-v_T\|_2^2+2\sum_{t=1}^{T-1}\eta_tR_{\cX}^2.
    \end{align*}
    Since $\lambda_1=0$, $\eta_t\geq\eta_{t+1}$, $\alpha_t\leq\alpha_{t+1}$, and $\|x-v_t\|_2\leq 2R_{\cX}$ for all $x\in\cX$ and $v_t\in\myoverline{\cX}$, we obtain
    \begin{align*}
        &\frac{\eta_{T-1}}{2}\left\|\frac{\lambda_T}{\eta_T}\right\|_2^2+\sum_{t=1}^{T-1}(g^\x_t)^\top(v_t-x) 
        \leq 2\left(\alpha_{T-1}+\sum_{t=1}^{T-1}\eta_t\right)R_{\cX}^2-\frac{\alpha_{T-1}}{2}\|x-v_T\|_2^2+\sum_{t=1}^{T-1}\frac{\|g^\x_t\|_2^2}{2\alpha_t}.
    \end{align*}
    Adding $(g^\x_T)^\top(v_T-x)$ to both sides and applying Young's inequality once more yields
    \begin{align*}
        &\frac{\eta_{T-1}}{2}\left\|\frac{\lambda_T}{\eta_T}\right\|_2^2+\sum_{t=1}^{T}(g^\x_t)^\top(v_t-x) \\
        &\leq 2\left(\alpha_{T-1}+\sum_{t=1}^{T-1}\eta_t\right)R_{\cX}^2+(g^\x_T)^\top(v_T-x)-\frac{\alpha_{T-1}}{2}\|x-v_T\|_2^2+\sum_{t=1}^{T-1}\frac{\|g^\x_t\|_2^2}{2\alpha_t} \\
        &\leq 2\left(\alpha_{T-1}+\sum_{t=1}^{T-1}\eta_t\right)R_{\cX}^2+\frac{\|g^\x_T\|_2^2}{2\alpha_{T-1}}+\sum_{t=1}^{T-1}\frac{\|g^\x_t\|_2^2}{2\alpha_t}.
    \end{align*}
    Let $\cF_t$ be the sigma algebra generated by all randomness up to the beginning of iteration $t$, and define $\bar g^\x_t:=\E[g^\x_t\mid\cF_t]$ and $\xi_t:=\bar g^\x_t-g^\x_t$. By Assumption~\ref{ass:stochastic-oracle}, $\bar g^\x_t\in\partial f_{z_t}(v_t)$. Adding $\sum_{t=1}^T\xi_t^\top(v_t-x)$ to both sides of the last display and using convexity of $f_{z_t}(\cdot)$, we get
    \begin{align*}
        &\frac{\eta_{T-1}}{2}\left\|\frac{\lambda_T}{\eta_T}\right\|_2^2+\sum_{t=1}^{T}\bigl(f(v_t,z_t)-f(x,z_t)\bigr) \\
        &\leq 2\left(\alpha_{T-1}+\sum_{t=1}^{T-1}\eta_t\right)R_{\cX}^2+\frac{\|g^\x_T\|_2^2}{2\alpha_{T-1}}+\sum_{t=1}^{T-1}\frac{\|g^\x_t\|_2^2}{2\alpha_t}+\sum_{t=1}^{T}\xi_t^\top(v_t-x).
    \end{align*}
    Taking the maximum over $x\in\cX$ gives
    \begin{align*}
        &\frac{\eta_{T-1}}{2}\left\|\frac{\lambda_T}{\eta_T}\right\|_2^2+\max_{x\in\cX}\sum_{t=1}^{T}\bigl(f(v_t,z_t)-f(x,z_t)\bigr) \\
        &\leq 2\left(\alpha_{T-1}+\sum_{t=1}^{T-1}\eta_t\right)R_{\cX}^2+\frac{\|g^\x_T\|_2^2}{2\alpha_{T-1}}+\sum_{t=1}^{T-1}\frac{\|g^\x_t\|_2^2}{2\alpha_t}+\max_{x\in\cX}\sum_{t=1}^{T}\xi_t^\top(v_t-x).
    \end{align*}
    It remains to control the martingale term. Since $x_1\in\cX$, for every $x\in\cX$,
    \begin{align*}
        \sum_{t=1}^{T}\xi_t^\top(v_t-x)
        &=\sum_{t=1}^{T}\xi_t^\top(v_t-x_1)+ \sum_{t=1}^{T}\xi_t^\top (x_1-x) 
        \leq \sum_{t=1}^{T}\xi_t^\top(v_t-x_1)+2R_{\cX}\left\|\sum_{t=1}^{T}\xi_t\right\|_2.
    \end{align*}
    Because $\E[\xi_t\mid\cF_t]=0$ and $v_t-x_1$ is $\cF_t$-measurable, $\E\left[\sum_{t=1}^{T}\xi_t^\top(v_t-x_1)\right]=0$.
    Moreover,
    \begin{align*}
        \textstyle \E\left[\left\|\sum_{t=1}^{T}\xi_t\right\|_2\right]
        \leq \left(\E\left[\left\|\sum_{t=1}^{T}\xi_t\right\|_2^2\right]\right)^{1/2} 
        =\left(\sum_{t=1}^{T}\E\|\xi_t\|_2^2\right)^{1/2} 
        \leq G_\x T^{1/2},
    \end{align*}
    where the last inequality follows from $\E\|\xi_t\|_2^2\leq\E\|g^\x_t\|_2^2\leq G_\x^2$. Hence,
    \begin{align}
    \label{eq:martingale}
        \textstyle \E\left[\max_{x\in\cX}\sum_{t=1}^{T}\xi_t^\top(v_t-x)\right]\leq 2G_\x R_{\cX}T^{1/2}.
    \end{align}
    Taking expectations and using $\E\|g^\x_t\|_2^2\leq G_\x^2$ for all $t$ finally yields~\eqref{eq:generic-primal-lmo-2}.
    This proves part~\ref{prop:generic-lmo:primal}, and the proof is complete by symmetry.
\end{proof}

\subsection{Proofs of \texorpdfstring{Proposition~\protect\ref{prop:generic-po}}{Proposition~2}}

To provide a unified theoretical framework, we extend the scope of \cref{prop:generic-po} to include the iterative structure of the PO-PO scheme in \cref{alg:po-po}. 

\vspace{-1.35em}
\begin{minipage}{0.55\textwidth}
    This algorithm can be viewed as an anytime variant of the stochastic subgradient descent-ascent algorithm introduced by \citet{nemirovski2009robust}. 
    Crucially, while the original results in \citep{nemirovski2009robust} are limited to the weak saddle gap guarantees, our generalized proposition allows us to establish the stronger saddle gap guarantee. 
    We note that the stronger saddle gap admits a natural generalization to variational inequalities, as discussed in \citep{juditsky2011solving}. 
    We provide the full statement and its proof below for completeness.
\end{minipage}
\hfill
\begin{minipage}{0.4\textwidth}
    \vspace{-1em} 
    \begin{algorithm}[H]
        \caption{\small \texttt{Nonsmooth PO-PO}}
        \label{alg:po-po}
        \begin{algorithmic}[1]
            \small
            \REQUIRE $\rho_t, \gamma_t > 0$ 
            \STATE \textbf{Init:} $x_1 \in \cX, y_1 \in \cY$
            \FOR{$t = 1, \dots, T$}
                \STATE $(g^\x_t, g^\y_t) = \mathcal{O}_f(x_t, y_t)$
                \STATE $\bar x_t = \frac{(t-1) \bar x_{t-1} + x_t}{t}$
                \STATE $x_{t+1} = \PO_{\mathcal{X}}(x_t - \rho_t g^\x_t)$
                \STATE $\bar y_t = \frac{(t-1) \bar y_{t-1} + y_t}{t}$
                \STATE $y_{t+1} = \PO_{\mathcal{Y}}(y_t - \gamma_t g^\y_t)$
            \ENDFOR
            \ENSURE $\bar x_T, \bar y_T$
        \end{algorithmic}
    \end{algorithm}
\end{minipage}
\medskip

\begin{proposition}[Generalization of Proposition~\ref{prop:generic-po}]
\label{prop:generic-po-restate}
    Suppose Assumptions~\ref{ass:extended} and \ref{ass:stochastic-oracle} hold.
    \begin{enumerate}[label=(\roman*),leftmargin=2em]
        \item \label{prop:generic-po:primal}
        Let $\{x_t\}_{t \geq 1}$ be the primal sequence generated by Algorithm~\ref{alg:po-lmo} or \ref{alg:po-po} with parameters satisfying $\rho_t \geq \rho_{t+1}$ for all $t \geq 1$. Define the dual evaluation sequence $\{z_t\}_{t \geq 1} = \{ u_t \}_{t \geq 1}$ for \cref{alg:po-lmo} and $\{z_t\}_{t \geq 1} = \{ y_t \}_{t \geq 1}$ for \cref{alg:po-po}. Then, for every $T \geq 1$,
        \begin{align}
        \label{eq:generic-primal-po-2}
            \E\left[\max_{x \in \cX}\sum_{t =1}^T \big( f(x_t, z_t) - f(x, z_t) \big) \right] &\leq \frac{2R_{\cX}^2}{\rho_T}+\sum_{t = 1}^T \frac{\rho_t G_\x^2}{2} + 2G_\x R_{\cX}T^{1/2}.
        \end{align}
        
        \item \label{prop:generic-po:dual}
        Let $\{y_t\}_{t \geq 1}$ be the dual sequences generated by Algorithm~\ref{alg:lmo-po} or \ref{alg:po-po} with parameters satisfying $\gamma_t \geq \gamma_{t+1}$ for all $t \geq 1$. Define the primal evaluation sequence $\{w_t\}_{t \geq 1} = \{ v_t \}_{t \geq 1}$ for \cref{alg:lmo-po} and $\{w_t\}_{t \geq 1} = \{ x_t \}_{t \geq 1}$ for \cref{alg:po-po}. Then, for every $T \geq 1$, 
        \begin{align}
        \label{eq:generic-dual-po-2}
            \E\left[\max_{y \in \cY}\sum_{t =1}^T \big( f(w_t,y)-f(w_t,y_t) \big) \right] &\leq \frac{2R_{\cY}^2}{\gamma_T}+\sum_{t = 1}^T \frac{\gamma_t G_\y^2}{2} + 2G_\y R_{\cY}T^{1/2}.
        \end{align}
    \end{enumerate}
\end{proposition}

\begin{proof}[Proof of \cref{prop:generic-po-restate}]
    By symmetry, it suffices to prove part~\ref{prop:generic-po:primal}. The proof of part~\ref{prop:generic-po:dual} follows from the same argument after exchanging
    \begin{align*}
        (x_t,g^\x_t,\rho_t,\cX,R_{\cX},G_\x,z_t)
        \quad \text{with} \quad
        (y_t,g^\y_t,\gamma_t,\cY,R_{\cY},G_\y,w_t),
    \end{align*}
    and replacing $f_{z_t}(\cdot)$ by $(-f)_{w_t}(\cdot)$. We therefore prove only the primal inequality.
    
    Fix any $x\in\cX$. The projected update gives
    \begin{align*}
        x_{t+1}=\PO_{\cX}(x_t-\rho_t g^\x_t), \qquad t=1,\ldots,T.
    \end{align*}
    By nonexpansiveness of the Euclidean projection,
    \begin{align*}
        \|x_{t+1}-x\|_2^2
        &\leq \|x_t-\rho_t g^\x_t-x\|_2^2 
        =\|x_t-x\|_2^2-2\rho_t(g^\x_t)^\top(x_t-x)+\rho_t^2\|g^\x_t\|_2^2.
    \end{align*}
    Rearranging yields
    \begin{align*}
        (g^\x_t)^\top(x_t-x)
        \leq \frac{\|x_t-x\|_2^2}{2\rho_t}-\frac{\|x_{t+1}-x\|_2^2}{2\rho_t}+\frac{\rho_t\|g^\x_t\|_2^2}{2}.
    \end{align*}
    Summing this inequality from $t=1$ to $T$ gives
    \begin{align*}
        \sum_{t=1}^{T}(g^\x_t)^\top(x_t-x)
        &\leq \frac{\|x_1-x\|_2^2}{2\rho_1}+\frac{1}{2}\sum_{t=2}^{T}\left(\frac{1}{\rho_t}-\frac{1}{\rho_{t-1}}\right)\|x_t-x\|_2^2-\frac{\|x_{T+1}-x\|_2^2}{2\rho_T}+\sum_{t=1}^{T}\frac{\rho_t\|g^\x_t\|_2^2}{2}.
    \end{align*}
    Since $\rho_t\geq \rho_{t+1}$ and $\|x_t-x\|_2\leq 2R_{\cX}$ for all $t$, we may drop the final negative term and obtain
    \begin{align*}
        \sum_{t=1}^{T}(g^\x_t)^\top(x_t-x)
        &\leq \frac{2R_{\cX}^2}{\rho_1}+2R_{\cX}^2\sum_{t=2}^{T}\left(\frac{1}{\rho_t}-\frac{1}{\rho_{t-1}}\right)+\sum_{t=1}^{T}\frac{\rho_t\|g^\x_t\|_2^2}{2} 
        =\frac{2R_{\cX}^2}{\rho_T}+\sum_{t=1}^{T}\frac{\rho_t\|g^\x_t\|_2^2}{2}.
    \end{align*}
    
    Let $\cF_t$ denote the sigma algebra generated by all randomness up to the beginning of iteration $t$, so that $x_t$ and $z_t$ are $\cF_t$-measurable. Define $\bar g^\x_t:=\E[g^\x_t\mid\cF_t] $ and $\xi_t:=\bar g^\x_t-g^\x_t$.
    By Assumption~\ref{ass:stochastic-oracle}, $\bar g^\x_t\in\partial f_{z_t}(x_t)$. Adding $\sum_{t=1}^{T}\xi_t^\top(x_t-x)$ to both sides of the last display and using convexity of $f_{z_t}(\cdot)$ gives, for every $x\in\cX$,
    \begin{align*}
        \sum_{t=1}^{T}\bigl(f(x_t,z_t)-f(x,z_t)\bigr)
        &\leq \sum_{t=1}^{T}(\bar g^\x_t)^\top(x_t-x) 
        \leq \frac{2R_{\cX}^2}{\rho_T}+\sum_{t=1}^{T}\frac{\rho_t\|g^\x_t\|_2^2}{2}+\sum_{t=1}^{T}\xi_t^\top(x_t-x).
    \end{align*}
    Taking the maximum over $x\in\cX$ yields
    \begin{align*}
        \max_{x\in\cX}\sum_{t=1}^{T}\bigl(f(x_t,z_t)-f(x,z_t)\bigr)
        \leq \frac{2R_{\cX}^2}{\rho_T}+\sum_{t=1}^{T}\frac{\rho_t\|g^\x_t\|_2^2}{2}+\max_{x\in\cX}\sum_{t=1}^{T}\xi_t^\top(x_t-x).
    \end{align*}
    The martingale term is controlled by the same argument used to derive~\eqref{eq:martingale}. Namely, replacing $v_t$ by $x_t$ gives
    \begin{align*}
        \textstyle\E\left[\max_{x\in\cX}\sum_{t=1}^{T}\xi_t^\top(x_t-x)\right]\leq 2G_\x R_{\cX}T^{1/2}.
    \end{align*}
    Taking expectations and using $\E\|g^\x_t\|_2^2\leq G_\x^2$ for all $t$ gives~\eqref{eq:generic-primal-po-2}.
    This proves part~\ref{prop:generic-po:primal}, and the proof is complete by symmetry.
\end{proof}

\subsection{Proofs of \texorpdfstring{Theorem~\protect\ref{theorem:convergence}\,\ref{theorem:convergence-lmo-lmo}}{Theorem~1 (i)}}

We first present the detailed version of Theorem~\ref{theorem:convergence}\,\ref{theorem:convergence-lmo-lmo}, including explicit constants and parameter choices.

\begin{theorem}[Detailed version of Theorem~\ref{theorem:convergence}\,\ref{theorem:convergence-lmo-lmo}]
\label{theorem:convergence-lmo-lmo-2}
    Suppose Assumptions~\ref{ass:extended} and~\ref{ass:stochastic-oracle} hold. Let $\{x_t,y_t\}_{t \geq 1}$ be the primal-dual sequence generated by \cref{alg:lmo-lmo} with parameters 
    \begin{align*}
        \alpha_t=C_1(t+1)^{1/2},
        \quad
        \beta_t=C_2(t+1)^{1/2},
        \quad
        \eta_t=C_3(t+1)^{-1/2},
        \quad
        \tau_t=C_4(t+1)^{-1/2}
    \end{align*}
    for some $C_1,C_2,C_3, C_4 >0$. 
    Then, for every $T\geq 2$,
    \begin{align*}
        \E \left[ \max_{y \in \cY}f(\bar{x}_T,y) - \min_{x \in \cX}f(x,\bar{y}_T) \right] 
        &\leq 
        \frac{
            2(C_1+2C_3)R_{\cX}^2
            +\frac{3G_\x^2}{2C_1}
            +\frac{G_\x^2}{2C_3}
            +2G_\x R_{\cX}
        }{T^{1/2}} \\
        &\quad+
        \frac{
            2(C_2+2C_4)R_{\cY}^2
            +\frac{3G_\y^2}{2C_2}
            +\frac{G_\y^2}{2C_4}
            +2G_\y R_{\cY}
        }{T^{1/2}}.
    \end{align*}
     If $C_1= C_3 =G_\x/R_{\cX}$ and $C_2 = C_4 = G_\y/R_{\cY}$ then the bound becomes
     \begin{align*}
        &\E\left[\max_{y\in\cY}f(\bar{x}_T,y)-\min_{x\in\cX}f(x,\bar{y}_T)\right] \leq \frac{10(G_{\x}R_{\cX}+G_{\y}R_{\cY})}{T^{1/2}}
    \end{align*}
\end{theorem}

The proof relies on the following lemma, which bounds the strong saddle gap for arbitrary monotone parameter sequences using \cref{prop:generic-lmo-restate}.

\begin{lemma} \label{cor:saddle-gap-lmo-lmo}
    Suppose Assumptions~\ref{ass:extended} and~\ref{ass:stochastic-oracle} hold. Let $\{x_t,y_t\}_{t\geq 1}$ be the primal-dual sequence generated by \cref{alg:lmo-lmo} with parameters satisfying
    $\alpha_t\leq\alpha_{t+1}$, $\beta_t\leq\beta_{t+1}$, $\eta_t\geq\eta_{t+1}$ and $\tau_t\geq\tau_{t+1}$ for all $t\geq 1$. Then, for every $T\geq 2$,
    \begin{align*}
        &\E\left[\max_{y\in\cY}f(\bar{x}_T,y)-\min_{x\in\cX}f(x,\bar{y}_T)\right] \\
        &\leq
        \frac{2}{T}\left(\alpha_{T-1}+\sum_{t=1}^{T-1}\eta_t\right)R_{\cX}^2
        +\frac{G_\x^2}{2\alpha_{T-1}T}
        +\frac{1}{T}\sum_{t=1}^{T-1}\frac{G_\x^2}{2\alpha_t}
        +\frac{G_\x^2}{2T\eta_{T-1}}
        +\frac{2G_\x R_{\cX}}{T^{1/2}} \\
        &\quad+
        \frac{2}{T}\left(\beta_{T-1}+\sum_{t=1}^{T-1}\tau_t\right)R_{\cY}^2
        +\frac{G_\y^2}{2\beta_{T-1}T}
        +\frac{1}{T}\sum_{t=1}^{T-1}\frac{G_\y^2}{2\beta_t}
        +\frac{G_\y^2}{2T\tau_{T-1}}
        +\frac{2G_\y R_{\cY}}{T^{1/2}}.
    \end{align*}
\end{lemma}

\begin{proof}[Proof of \cref{cor:saddle-gap-lmo-lmo}]
    Define $\bar v_T:= \frac{1}{T} \sum_{t=1}^{T}v_t$ and $\bar u_T:= \frac{1}{T} \sum_{t=1}^{T}u_t$. Since $v_1=x_1$ and $u_1=y_1$, the LMO recursions imply
    \begin{align*}
        \begin{cases}
            \displaystyle \frac{\lambda_T}{\eta_T}
            = 
            \frac{\lambda_{T-1}}{\eta_{T-1}} + v_T - x_T 
            =
            \frac{\lambda_0}{\eta_0} + (v_T - x_T) + \dots + (v_0 - x_0)
            =
            \sum_{t=1}^{T}(v_t-x_t)
            =
            T(\bar v_T-\bar x_T), \\
            \displaystyle \frac{\mu_T}{\tau_T}
            =
            \frac{\mu_{T-1}}{\tau_{T-1}} + u_T - y_T 
            =
            \frac{\mu_{0}}{\tau_{0}}+(u_T - y_T) + \dots + (u_1-y_1)
            =
            \sum_{t=1}^{T}(u_t-y_t)
            =
            T(\bar u_T-\bar y_T).
        \end{cases}
    \end{align*}
    We first convert the primal regret bound in \cref{prop:generic-lmo-restate} into a bound with comparator $\bar y_T$. By convexity of $(-f)_x(\cdot)$, we have $\frac{1}{T}\sum_{t=1}^{T}-f(x,u_t)\geq -f(x,\bar u_T)$.
    Moreover, $(-f)_x(\cdot)$ is $G_\y$-Lipschitz on $\myoverline{\cY}$, so
    \begin{align*}
        -f(x,\bar u_T)
        \geq
        -f(x,\bar y_T)-G_\y\|\bar u_T-\bar y_T\|_2
        =
        -f(x,\bar y_T)-\frac{G_\y}{T}\left\|\frac{\mu_T}{\tau_T}\right\|_2.
    \end{align*}
    Therefore, applying \cref{prop:generic-lmo-restate}\,\ref{prop:generic-lmo:primal} with $z_t=u_t$ gives
    \begin{equation}
    \label{eq:primal-gap-lmo-lmo}
    \begin{aligned}
        &\E\left[
            \frac{\eta_{T-1}}{2T}\left\|\frac{\lambda_T}{\eta_T}\right\|_2^2
            -\frac{G_\y}{T}\left\|\frac{\mu_T}{\tau_T}\right\|_2
            +\frac{1}{T}\sum_{t=1}^{T}f(v_t,u_t)
            -\min_{x\in\cX}f(x,\bar y_T)
        \right] \\
        &\leq
        \frac{2}{T}\left(\alpha_{T-1}+\sum_{t=1}^{T-1}\eta_t\right)R_{\cX}^2
        +\frac{G_\x^2}{2\alpha_{T-1}T}
        +\frac{1}{T}\sum_{t=1}^{T-1}\frac{G_\x^2}{2\alpha_t}
        +\frac{2G_\x R_{\cX}}{T^{1/2}}.
    \end{aligned}
    \end{equation}
    Similarly, by convexity of $f_y(\cdot)$, we have $\frac{1}{T}\sum_{t=1}^{T}f(v_t,y)\geq f(\bar v_T,y)$.
    Adding~\eqref{eq:primal-gap-lmo-lmo} and~\eqref{eq:dual-gap-lmo-lmo} cancels the average payoff term. It remains to lower bound the two quadratic expressions. For any $a>0$, $b\geq 0$, and $s\geq 0$, we use $as^2/2-bs\geq -b^2/(2a)$. Hence,
    \begin{align*}
        \frac{\eta_{T-1}}{2T}\left\|\frac{\lambda_T}{\eta_T}\right\|_2^2
        -\frac{G_\x}{T}\left\|\frac{\lambda_T}{\eta_T}\right\|_2
        &\geq
        -\frac{G_\x^2}{2T\eta_{T-1}}, \\
        \frac{\tau_{T-1}}{2T}\left\|\frac{\mu_T}{\tau_T}\right\|_2^2
        -\frac{G_\y}{T}\left\|\frac{\mu_T}{\tau_T}\right\|_2
        &\geq
        -\frac{G_\y^2}{2T\tau_{T-1}}.
    \end{align*}
    Combining these bounds with~\eqref{eq:primal-gap-lmo-lmo} and~\eqref{eq:dual-gap-lmo-lmo} proves the claim.
    Since $f_y(\cdot)$ is $G_\x$-Lipschitz on $\myoverline{\cX}$, we also have
    \begin{align*}
        f(\bar v_T,y)
        \geq
        f(\bar x_T,y)-G_\x\|\bar v_T-\bar x_T\|_2
        =
        f(\bar x_T,y)-\frac{G_\x}{T}\left\|\frac{\lambda_T}{\eta_T}\right\|_2.
    \end{align*}
    Applying \cref{prop:generic-lmo-restate}\,\ref{prop:generic-lmo:dual} with $w_t=v_t$ gives
    \begin{equation}
    \label{eq:dual-gap-lmo-lmo}
    \begin{aligned}
        &\E\left[
            \frac{\tau_{T-1}}{2T}\left\|\frac{\mu_T}{\tau_T}\right\|_2^2
            -\frac{G_\x}{T}\left\|\frac{\lambda_T}{\eta_T}\right\|_2
            +\max_{y\in\cY}f(\bar x_T,y)
            -\frac{1}{T}\sum_{t=1}^{T}f(v_t,u_t)
        \right] \\
        &\leq
        \frac{2}{T}\left(\beta_{T-1}+\sum_{t=1}^{T-1}\tau_t\right)R_{\cY}^2
        +\frac{G_\y^2}{2\beta_{T-1}T}
        +\frac{1}{T}\sum_{t=1}^{T-1}\frac{G_\y^2}{2\beta_t}
        +\frac{2G_\y R_{\cY}}{T^{1/2}}.
    \end{aligned}
    \end{equation}
    Adding~\eqref{eq:primal-gap-lmo-lmo} and~\eqref{eq:dual-gap-lmo-lmo} cancels the average payoff term. It remains to lower bound the two quadratic expressions. For any $a>0$, $b\geq 0$, and $s\geq 0$, we use $as^2/2-bs\geq -b^2/(2a)$. Hence,
    \begin{align*}
        \frac{\eta_{T-1}}{2T}\left\|\frac{\lambda_T}{\eta_T}\right\|_2^2
        -\frac{G_\x}{T}\left\|\frac{\lambda_T}{\eta_T}\right\|_2
        \geq
        -\frac{G_\x^2}{2T\eta_{T-1}}, \qquad
        \frac{\tau_{T-1}}{2T}\left\|\frac{\mu_T}{\tau_T}\right\|_2^2
        -\frac{G_\y}{T}\left\|\frac{\mu_T}{\tau_T}\right\|_2
        \geq
        -\frac{G_\y^2}{2T\tau_{T-1}}.
    \end{align*}
    Combining these bounds with~\eqref{eq:primal-gap-lmo-lmo} and~\eqref{eq:dual-gap-lmo-lmo} proves the claim.
\end{proof}

\begin{proof}[Proof of \cref{theorem:convergence-lmo-lmo-2}]
    The parameter choices satisfy the monotonicity conditions in \cref{cor:saddle-gap-lmo-lmo}. Moreover, for every $T\geq 2$,
    \begin{align*}
        \alpha_{T-1}=C_1T^{1/2},
        \qquad
        \beta_{T-1}=C_2T^{1/2},
        \qquad
        \eta_{T-1}=C_3T^{-1/2},
        \qquad
        \tau_{T-1}=C_4T^{-1/2}.
    \end{align*}
    We also use
    \begin{align*}
        \sum_{t=1}^{T-1}\eta_t
        \leq 2C_3T^{1/2},
        \qquad
        \sum_{t=1}^{T-1}\tau_t
        \leq 2C_4T^{1/2},
        \qquad
        \sum_{t=1}^{T-1}\frac{1}{\alpha_t}
        \leq \frac{2T^{1/2}}{C_1},
        \qquad
        \sum_{t=1}^{T-1}\frac{1}{\beta_t}
        \leq \frac{2T^{1/2}}{C_2}.
    \end{align*}
    Substituting these bounds into \cref{cor:saddle-gap-lmo-lmo} gives the desired bound.
\end{proof}

\subsection{Proofs of \texorpdfstring{Theorem~\protect\ref{theorem:convergence}\,\ref{theorem:convergence-lmo-po} and \ref{theorem:convergence-po-lmo}}{Theorem~1 (ii) and (iii)}}

We first present the detailed versions of Theorem~\ref{theorem:convergence}\,\ref{theorem:convergence-lmo-po} and~\ref{theorem:convergence-po-lmo}, including explicit constants and parameter choices. 

\begin{theorem}[Detailed version of Theorem~\ref{theorem:convergence}\,\ref{theorem:convergence-lmo-po}]
\label{theorem:convergence-lmo-po-2}
    Suppose Assumptions~\ref{ass:extended} and~\ref{ass:stochastic-oracle} hold. Let $\{x_t,y_t\}_{t \geq 1}$ be the primal-dual sequence generated by \cref{alg:lmo-po} with parameters  
    \begin{align*}
        \alpha_t=C_1(t+1)^{1/2},
        \quad
        \gamma_t=C_2t^{-1/2},
        \quad
        \eta_t=C_3(t+1)^{-1/2}
    \end{align*}
    for some $C_1, C_2, C_3 > 0$. 
    Then, for every $T\geq 2$,
    \begin{align*}
        \E \left[ \max_{y \in \cY}f(\bar{x}_T,y) - \min_{x \in \cX}f(x,\bar{y}_T) \right] 
        &\leq
        \frac{
            2(C_1+2C_3)R_{\cX}^2
            +\frac{3G_\x^2}{2C_1}
            +\frac{G_\x^2}{2C_3}
            +2G_\x R_{\cX}
        }{T^{1/2}} \\
        &\quad+ 
        \frac{
            \frac{2R_{\cY}^2}{C_2}
            +C_2G_\y^2
            +2G_\y R_{\cY}
        }{T^{1/2}}. 
    \end{align*}
    If $C_1= C_3 =G_\x/R_{\cX}$ and $C_2 = G_\y/R_{\cY}$ then the bound becomes
     \begin{align*}
        &\E\left[\max_{y\in\cY}f(\bar{x}_T,y)-\min_{x\in\cX}f(x,\bar{y}_T)\right] \leq \frac{10G_{\x}R_{\cX}+5G_{\y}R_{\cY}}{T^{1/2}}
    \end{align*}
\end{theorem} 

\begin{theorem}[Detailed version of Theorem~\ref{theorem:convergence}\,\ref{theorem:convergence-po-lmo}]
\label{theorem:convergence-po-lmo-2}
    Suppose Assumptions~\ref{ass:extended} and~\ref{ass:stochastic-oracle} hold. Let $\{x_t,y_t\}_{t \geq 1}$ be the primal-dual sequence generated by \cref{alg:po-lmo} with parameters 
    \begin{align*}
        \beta_t= C_1 (t+1)^{1/2}, \quad \rho_t= C_2 (t+1)^{-1/2}, \quad \tau_t =C_3 (t+1)^{-1/2}
    \end{align*}
    for some $C_1, C_2, C_3 > 0$.
    Then, for every $T \geq 2$,
    \begin{align*}
        \E \left[ \max_{y \in \cY}f(\bar{x}_T,y) - \min_{x \in \cX}f(x,\bar{y}_T) \right] 
        &\leq
        \frac{
            \frac{2R_{\cX}^2}{C_2}
            +C_2G_\x^2
            +2G_\x R_{\cX}
        }{T^{1/2}} \\
        &\quad+
        \frac{
            2(C_1+2C_3)R_{\cY}^2
            +\frac{3G_\y^2}{2C_1}
            +\frac{G_\y^2}{2C_3}
            +2G_\y R_{\cY}
        }{T^{1/2}}.
    \end{align*}
    If $C_1= C_3 =G_\y/R_{\cY}$ and $C_2 = G_\x/R_{\cX}$ then the bound becomes
     \begin{align*}
        &\E\left[\max_{y\in\cY}f(\bar{x}_T,y)-\min_{x\in\cX}f(x,\bar{y}_T)\right] \leq \frac{5G_{\x}R_{\cX}+10G_{\y}R_{\cY}}{T^{1/2}}
    \end{align*}
\end{theorem} 

The proof relies on the following lemma, which bounds the strong saddle gap for arbitrary monotone parameter sequences using \cref{prop:generic-lmo-restate,prop:generic-po-restate}.

\begin{lemma}
\label{cor:saddle-gap-lmo-po}
    Suppose Assumptions~\ref{ass:extended} and~\ref{ass:stochastic-oracle} hold.
    \begin{enumerate}[label=(\roman*),leftmargin=2em]
        \item\label{cor:saddle-gap-lmo-po:lmo-po}
        Let $\{x_t,y_t\}_{t\geq 1}$ be the primal-dual sequence generated by \cref{alg:lmo-po} with parameters satisfying $\alpha_t\leq\alpha_{t+1}$, $\eta_t\geq\eta_{t+1}$, and $\gamma_t\geq\gamma_{t+1}$ for all $t\geq 1$. Then, for every $T\geq 2$,
        \begin{align*}
            &\E\left[\max_{y\in\cY}f(\bar{x}_T,y)-\min_{x\in\cX}f(x,\bar{y}_T)\right] \\
            &\leq
            \frac{2}{T}\left(\alpha_{T-1}+\sum_{t=1}^{T-1}\eta_t\right)R_{\cX}^2
            +\frac{G_\x^2}{2\alpha_{T-1}T}
            +\frac{1}{T}\sum_{t=1}^{T-1}\frac{G_\x^2}{2\alpha_t}
            +\frac{G_\x^2}{2T\eta_{T-1}}
            +\frac{2G_\x R_{\cX}}{T^{1/2}} \\
            &\quad+
            \frac{2R_{\cY}^2}{\gamma_TT}
            +\frac{1}{T}\sum_{t=1}^{T}\frac{\gamma_tG_\y^2}{2}
            +\frac{2G_\y R_{\cY}}{T^{1/2}}.
        \end{align*}
        \item\label{cor:saddle-gap-lmo-po:po-lmo}
        Let $\{x_t,y_t\}_{t\geq 1}$ be the primal-dual sequence generated by \cref{alg:po-lmo} with parameters satisfying $\rho_t\geq\rho_{t+1}$, $\beta_t\leq\beta_{t+1}$, and $\tau_t\geq\tau_{t+1}$ for all $t\geq 1$. Then, for every $T\geq 2$,
        \begin{align*}
            &\E\left[\max_{y\in\cY}f(\bar{x}_T,y)-\min_{x\in\cX}f(x,\bar{y}_T)\right] \\
            &\leq
            \frac{2R_{\cX}^2}{\rho_TT}
            +\frac{1}{T}\sum_{t=1}^{T}\frac{\rho_tG_\x^2}{2}
            +\frac{2G_\x R_{\cX}}{T^{1/2}} \\
            &\quad+
            \frac{2}{T}\left(\beta_{T-1}+\sum_{t=1}^{T-1}\tau_t\right)R_{\cY}^2
            +\frac{G_\y^2}{2\beta_{T-1}T}
            +\frac{1}{T}\sum_{t=1}^{T-1}\frac{G_\y^2}{2\beta_t}
            +\frac{G_\y^2}{2T\tau_{T-1}}
            +\frac{2G_\y R_{\cY}}{T^{1/2}}.
        \end{align*}
    \end{enumerate}
\end{lemma}

\begin{proof}[Proof of \cref{cor:saddle-gap-lmo-po}]
    By symmetry, it suffices to prove part~\ref{cor:saddle-gap-lmo-po:lmo-po}. The proof of part~\ref{cor:saddle-gap-lmo-po:po-lmo} follows from the same argument after exchanging the primal and dual roles.
    
    Define $\bar v_T:=\frac{1}{T}\sum_{t=1}^{T}v_t$. Since $v_1=x_1$, the LMO recursion implies
    \begin{align*}
        \frac{\lambda_T}{\eta_T}
        =
        \sum_{t=1}^{T}(v_t-x_t)
        =
        T(\bar v_T-\bar x_T).
    \end{align*}
    We first control the primal side. By convexity of $(-f)_x(\cdot)$, for every $x\in\cX$, we have $\frac{1}{T}\sum_{t=1}^{T}-f(x,y_t)\geq -f(x,\bar y_T)$.
    Applying \cref{prop:generic-lmo-restate}\,\ref{prop:generic-lmo:primal} with $z_t=y_t$ gives
    \begin{equation}
    \label{eq:primal-gap-lmo-po}
    \begin{aligned}
        &\E\left[
            \frac{\eta_{T-1}}{2T}\left\|\frac{\lambda_T}{\eta_T}\right\|_2^2
            +
            \frac{1}{T}\sum_{t=1}^{T}f(v_t,y_t)
            -
            \min_{x\in\cX}f(x,\bar y_T)
        \right] \\
        &\leq
        \frac{2}{T}\left(\alpha_{T-1}+\sum_{t=1}^{T-1}\eta_t\right)R_{\cX}^2
        +\frac{G_\x^2}{2\alpha_{T-1}T}
        +\frac{1}{T}\sum_{t=1}^{T-1}\frac{G_\x^2}{2\alpha_t}
        +\frac{2G_\x R_{\cX}}{T^{1/2}}.
    \end{aligned}
    \end{equation}
    We next control the dual side. By convexity of $f_y(\cdot)$, we have $\frac{1}{T}\sum_{t=1}^{T}f(v_t,y)\geq f(\bar v_T,y)$.
    Since $f_y(\cdot)$ is $G_\x$-Lipschitz on $\myoverline{\cX}$,
    \begin{align*}
        f(\bar v_T,y)
        \geq
        f(\bar x_T,y)-G_\x\|\bar v_T-\bar x_T\|_2
        =
        f(\bar x_T,y)-\frac{G_\x}{T}\left\|\frac{\lambda_T}{\eta_T}\right\|_2.
    \end{align*}
    Applying \cref{prop:generic-po-restate}\,\ref{prop:generic-po:dual} with $w_t=v_t$ gives
    \begin{equation}
    \label{eq:dual-gap-lmo-po}
    \begin{aligned}
        &\E\left[
            -\frac{G_\x}{T}\left\|\frac{\lambda_T}{\eta_T}\right\|_2
            +
            \max_{y\in\cY}f(\bar x_T,y)
            -
            \frac{1}{T}\sum_{t=1}^{T}f(v_t,y_t)
        \right] \\
        &\leq
        \frac{2R_{\cY}^2}{\gamma_TT}
        +\frac{1}{T}\sum_{t=1}^{T}\frac{\gamma_tG_\y^2}{2}
        +\frac{2G_\y R_{\cY}}{T^{1/2}}.
    \end{aligned}
    \end{equation}
    Adding~\eqref{eq:primal-gap-lmo-po} and~\eqref{eq:dual-gap-lmo-po} cancels the average payoff term. It remains to lower bound the quadratic expression in $\lambda_T$. For any $a>0$, $b\geq 0$, and $s\geq 0$, we use $as^2/2-bs\geq -b^2/(2a)$. Hence,
    \begin{align*}
        \frac{\eta_{T-1}}{2T}\left\|\frac{\lambda_T}{\eta_T}\right\|_2^2
        -
        \frac{G_\x}{T}\left\|\frac{\lambda_T}{\eta_T}\right\|_2
        \geq
        -\frac{G_\x^2}{2T\eta_{T-1}}.
    \end{align*}
    Combining this bound with~\eqref{eq:primal-gap-lmo-po} and~\eqref{eq:dual-gap-lmo-po} proves part~\ref{cor:saddle-gap-lmo-po:lmo-po}, and therefore completes the proof.
\end{proof}

\begin{proof}[Proof of \cref{theorem:convergence-lmo-po-2,theorem:convergence-po-lmo-2}]
    We prove the two detailed theorems by substituting the stated parameter choices into \cref{cor:saddle-gap-lmo-po}. For the LMO-PO case, the parameters satisfy
    \begin{align*}
        \alpha_{T-1}=C_1T^{1/2},
        \qquad
        \eta_{T-1}=C_3T^{-1/2},
        \qquad
        \gamma_T=C_2T^{-1/2}.
    \end{align*}
    Moreover,
    \begin{align*}
        \sum_{t=1}^{T-1}\eta_t
        \leq 2C_3T^{1/2},
        \qquad
        \sum_{t=1}^{T-1}\frac{1}{\alpha_t}
        \leq \frac{2T^{1/2}}{C_1},
        \qquad
        \sum_{t=1}^{T}\gamma_t
        \leq 2C_2T^{1/2}.
    \end{align*}
    Substituting these estimates into \cref{cor:saddle-gap-lmo-po}\,\ref{cor:saddle-gap-lmo-po:lmo-po} gives the desired bound in \cref{theorem:convergence-lmo-po-2}.
    
    The proof of \cref{theorem:convergence-po-lmo-2} is identical, using \cref{cor:saddle-gap-lmo-po}\,\ref{cor:saddle-gap-lmo-po:po-lmo} together with
    \begin{align*}
        \rho_T=C_2T^{-1/2},
        \qquad
        \beta_{T-1}=C_1T^{1/2},
        \qquad
        \tau_{T-1}=C_3T^{-1/2},
    \end{align*}
    and the analogous bounds
    \begin{align*}
        \sum_{t=1}^{T}\rho_t
        \leq 2C_2T^{1/2},
        \qquad
        \sum_{t=1}^{T-1}\tau_t
        \leq 2C_3T^{1/2},
        \qquad
        \sum_{t=1}^{T-1}\frac{1}{\beta_t}
        \leq \frac{2T^{1/2}}{C_1}.
    \end{align*}
    gives the desired bound in \cref{theorem:convergence-po-lmo-2}.
\end{proof}

\subsection{Strong Saddle-Gap Guarantee for the PO-PO Baseline}

We first state the detailed convergence guarantee for \cref{alg:po-po}. This result is not part of \cref{theorem:convergence}, but it follows from the same framework and gives the standard PO-PO benchmark in the strong saddle-gap metric.

\begin{theorem} \label{theorem:convergence-po-po-2}
    Suppose Assumptions~\ref{ass:extended} and~\ref{ass:stochastic-oracle} hold. Let $\{x_t,y_t\}_{t\geq 1}$ be the primal-dual sequence generated by \cref{alg:po-po} with parameters
    \begin{align*}
        \rho_t=C_1 t^{-1/2},
        \qquad
        \gamma_t=C_2t^{-1/2}
    \end{align*}
    for some $C_1, C_2 > 0$.
    Then, for every $T\geq 1$,
        \begin{align*}
            \E \left[ \max_{y \in \cY}f(\bar{x}_T,y) - \min_{x \in \cX}f(x,\bar{y}_T) \right]
            \leq 
            \frac{
                \frac{2R_{\cX}^2}{C_1}
                +C_1G_\x^2
                +2G_\x R_{\cX}
            }{T^{1/2}}
            + 
            \frac{
                \frac{2R_{\cY}^2}{C_2}
                +C_2G_\y^2
                +2G_\y R_{\cY}
            }{T^{1/2}}.
        \end{align*}
        If $C_1=G_\x/R_{\cX}$ and $C_2 = G_\y/R_{\cY}$ then the bound becomes
     \begin{align*}
        &\E\left[\max_{y\in\cY}f(\bar{x}_T,y)-\min_{x\in\cX}f(x,\bar{y}_T)\right] \leq \frac{5(G_{\x}R_{\cX}+G_{\y}R_{\cY})}{T^{1/2}}
    \end{align*}
\end{theorem} 

The proof relies on the following lemma, which bounds the strong saddle gap for arbitrary monotone parameter sequences using \cref{prop:generic-po-restate}.

\begin{lemma} \label{cor:saddle-gap-po-po}
    Suppose Assumptions~\ref{ass:extended} and~\ref{ass:stochastic-oracle} hold. Let $\{x_t,y_t\}_{t\geq 1}$ be the primal-dual sequence generated by \cref{alg:po-po} with parameters satisfying $\rho_t\geq\rho_{t+1}$ and $\gamma_t\geq\gamma_{t+1}$ for all $t\geq 1$. Then, for every $T\geq 1$,
    \begin{align*}
        &\E\left[\max_{y\in\cY}f(\bar{x}_T,y)-\min_{x\in\cX}f(x,\bar{y}_T)\right] \\
        &\leq
        \frac{2R_{\cX}^2}{\rho_TT}
        +\frac{1}{T}\sum_{t=1}^{T}\frac{\rho_tG_\x^2}{2}
        +\frac{2G_\x R_{\cX}}{T^{1/2}}
        +\frac{2R_{\cY}^2}{\gamma_TT}
        +\frac{1}{T}\sum_{t=1}^{T}\frac{\gamma_tG_\y^2}{2}
        +\frac{2G_\y R_{\cY}}{T^{1/2}}.
    \end{align*}
\end{lemma}

\begin{proof}[Proof of \cref{cor:saddle-gap-lmo-po}]
    We first control the primal side. By convexity of $(-f)_x(\cdot)$, for every $x\in\cX$, we have $\frac{1}{T}\sum_{t=1}^{T}f(x,y_t)\leq f(x,\bar y_T)$.
    Therefore,
    \begin{align*}
        \max_{x\in\cX} \, \frac{1}{T}\sum_{t=1}^{T}f(x_t,y_t)-f(x,\bar y_T) 
        \leq
        \frac{1}{T}\max_{x\in\cX}\sum_{t=1}^{T}\bigl(f(x_t,y_t)-f(x,y_t)\bigr).
    \end{align*}
    Applying \cref{prop:generic-po-restate}\,\ref{prop:generic-po:primal} with $z_t=y_t$ gives
    \begin{equation}
    \label{eq:primal-gap-po-po}
    \begin{aligned}
        \E\left[
            \frac{1}{T}\sum_{t=1}^{T}f(x_t,y_t)-\min_{x\in\cX}f(x,\bar y_T)
        \right] 
        \leq
        \frac{2R_{\cX}^2}{\rho_TT}
        +\frac{1}{T}\sum_{t=1}^{T}\frac{\rho_tG_\x^2}{2}
        +\frac{2G_\x R_{\cX}}{T^{1/2}}.
    \end{aligned}
    \end{equation}
    We next control the dual side. By convexity of $f(\cdot,y)$, for every $y\in\cY$, we have $f(\bar x_T,y)\leq \frac{1}{T}\sum_{t=1}^{T}f(x_t,y)$.
    Hence,
    \begin{align*}
        \max_{y\in\cY}f(\bar x_T,y)-\frac{1}{T}\sum_{t=1}^{T}f(x_t,y_t)
        &\leq
        \frac{1}{T}\max_{y\in\cY}\sum_{t=1}^{T}\bigl(f(x_t,y)-f(x_t,y_t)\bigr).
    \end{align*}
    Applying \cref{prop:generic-po-restate}\,\ref{prop:generic-po:dual} with $w_t=x_t$ gives
    \begin{equation}
    \label{eq:dual-gap-po-po}
    \begin{aligned}
        \E\left[
            \max_{y\in\cY}f(\bar x_T,y)-\frac{1}{T}\sum_{t=1}^{T}f(x_t,y_t)
        \right] 
        \leq
        \frac{2R_{\cY}^2}{\gamma_TT}
        +\frac{1}{T}\sum_{t=1}^{T}\frac{\gamma_tG_\y^2}{2}
        +\frac{2G_\y R_{\cY}}{T^{1/2}}.
    \end{aligned}
    \end{equation}
    Adding~\eqref{eq:primal-gap-po-po} and~\eqref{eq:dual-gap-po-po} cancels the average payoff term and proves the lemma.
\end{proof}

\begin{proof}[Proof of \cref{theorem:convergence-po-po-2}]
    The parameter choices satisfy the monotonicity conditions in \cref{cor:saddle-gap-po-po}. Moreover,
    \begin{align*}
        \rho_T=C_1T^{-1/2},
        \qquad
        \gamma_T=C_2T^{-1/2}.
    \end{align*}
    We also use
    \begin{align*}
        \sum_{t=1}^{T}\rho_t
        \leq 2C_1T^{1/2},
        \qquad
        \sum_{t=1}^{T}\gamma_t
        \leq 2C_2T^{1/2}.
    \end{align*}
    Substituting these estimates into \cref{cor:saddle-gap-po-po} gives the desired bound.
\end{proof}

\subsection{Proofs of \texorpdfstring{Theorem~\protect\ref{thm:lower-bound-lmo-lmo}}{Theorem~2}}

For the convenience, we repeat the statement of~\cref{thm:lower-bound-lmo-lmo}.

\begin{theorem}[Restatement of Theorem~\ref{thm:lower-bound-lmo-lmo}]
\label{thm:lower-bound-lmo-lmo-repeat}
    Let $G_\x>0,G_\y>0, R_{\cX} >0, R_{\cY}>0$ be fixed. Then, for any generic LMO-LMO method of the form \cref{alg:generic-lmo-lmo} running for $T \geq 1$ iterations, there exist convex compact subsets $\cX \subseteq \R^n, \cY \subseteq \R^m$ with radii $R_{\cX}$ and $R_{\cY}$, respectively, and a function $f \in \cF_{G_\x,G_\y,\|\cdot\|_2}^0(\cX,\cY)$ with $T \leq  \min\{m,n\}/4-1$ such that
    \begin{align*}
        \max_{y\in\cY}f(\bar x_T,y)-\min_{x\in\cX}f(x,\bar y_T)
        \geq
        \frac{G_\x R_{\cX}}{2(T+1)^{1/2}}
        +
        \frac{G_\y R_{\cY}}{2(T+1)^{1/2}}.
    \end{align*}
\end{theorem}

\begin{proof}[Proof of \cref{thm:lower-bound-lmo-lmo-repeat}]
    Fix $T\geq 1$, and choose dimensions $n,m$ such that $n,m\geq 4(T+1)$. Let
    \begin{align*}
        \cX:=R_{\cX} \cdot \Delta_n,
        \qquad
        \cY:=R_{\cY} \cdot \Delta_m,
    \end{align*}
    where $\Delta_k:=\{z\in\R_+^k:\sum_{i=1}^{k}z_i=1\}$. The adversary initializes the method at the extreme points $x_1=R_{\cX} \cdot e_1$ and $y_1=R_{\cY} \cdot e_1$, and fixes LMO tie-breaking rules that always return extreme-point minimizers. This is valid because a linear objective over a simplex always admits an extreme-point minimizer.
    Consider the payoff function
    \begin{align*}
        f(x,y):=G_\x\|x\|_2-G_\y\|y\|_2.
    \end{align*}
    The function $f$ is convex in $x$ and concave in $y$. 
    It is also easy to verify that $f\in\cF_{G_\x,G_\y,\|\cdot\|_2}^0(\cX,\cY)$.
    
    Since $x_1$ is an extreme point of $\cX$ and each primal LMO call returns an extreme point of $\cX$, all points $x_1,\ldots,x_{T+1}$ are vertices of $R_{\cX} \cdot \Delta_n$. By \cref{alg:generic-lmo-lmo}, the output satisfies
    \begin{align*}
        \bar x_T\in\conv\{x_1,\ldots,x_{T+1}\}.
    \end{align*}
    Therefore, there exist distinct indices $i_1,\ldots,i_{k_T}\in[n]$ with $k_T\leq T+1$ such that
    \begin{align*}
        \bar x_T\in R_{\cX} \cdot \conv \{e_{i_1},\ldots,e_{i_{k_T}}\}.
    \end{align*}
    Consequently, we have
    \begin{align*}
        \|\bar x_T\|_2
        \geq
        \min\left\{\|x\|_2:x\in R_{\cX} \cdot \conv\{e_{i_1},\ldots,e_{i_{k_T}}\}\right\} 
        =
        \frac{R_{\cX}}{k_T^{1/2}}
        \geq
        \frac{R_{\cX}}{(T+1)^{1/2}}.
    \end{align*}
    Similarly, since $y_1$ is an extreme point of $\cY$ and each dual LMO call returns an extreme point of $\cY$, all points $y_1,\ldots,y_{T+1}$ are vertices of $R_{\cY} \cdot \Delta_m$, we have
    \begin{align*}
        \|\bar y_T\|_2
        \geq
        \frac{R_{\cY}}{(T+1)^{1/2}}.
    \end{align*}
    On the other hand, over the full simplices,
    \begin{align*}
        \min_{x\in\cX}\|x\|_2=\frac{R_{\cX}}{n^{1/2}},
        \qquad
        \min_{y\in\cY}\|y\|_2=\frac{R_{\cY}}{m^{1/2}}.
    \end{align*}
    Therefore,
    \begin{align*}
        \max_{y\in\cY}f(\bar x_T,y)-\min_{x\in\cX}f(x,\bar y_T) 
        &=
        G_\x\|\bar x_T\|_2
        -G_\y\min_{y\in\cY}\|y\|_2
        -G_\x\min_{x\in\cX}\|x\|_2
        +G_\y\|\bar y_T\|_2 \\
        &\geq
        \frac{G_\x R_{\cX}}{(T+1)^{1/2}}
        -\frac{G_\x R_{\cX}}{n^{1/2}}
        +
        \frac{G_\y R_{\cY}}{(T+1)^{1/2}}
        -\frac{G_\y R_{\cY}}{m^{1/2}}.
    \end{align*}
    Since $n,m\geq 4(T+1)$, we have
    \begin{align*}
        \frac{1}{n^{1/2}}\leq \frac{1}{2(T+1)^{1/2}},
        \qquad
        \frac{1}{m^{1/2}}\leq \frac{1}{2(T+1)^{1/2}}.
    \end{align*}
    Combining the last two displays gives the desired bound.
\end{proof}

\subsection{Proofs of \texorpdfstring{Theorem~\protect\ref{thm:lower-bound-lmo-po}}{Theorem~3}}

For the convenience, we repeat the statement of~\cref{thm:lower-bound-lmo-po}.

\begin{theorem} [Restatement of Theorem~\ref{thm:lower-bound-lmo-po}]
\label{thm:lower-bound-lmo-po-repeat}
    Let $G_\x>0,G_\y>0, R_{\cX} >0, R_{\cY}>0$ be fixed. Then, for any generic LMO-PO method of the form \cref{alg:generic-lmo-po} running for $T \geq 1$ iterations, there exist convex compact subsets $\cX \subseteq \R^n, \cY \subseteq \R^m$ with radii $R_{\cX}$ and $R_{\cY}$, respectively, and a function $f \in \cF_{G_\x,G_\y,\|\cdot\|_2}^0(\cX,\cY)$ with $T \leq  \min\{m,n\}/4-1$ such that
    \begin{align*}
        \max_{y \in \cY} f(\bar{x}_T,y)-\min_{x \in \cX} f(x,\bar{y}_T)
        \geq
        \frac{G_\x R_{\cX}}{2(T+1)^{1/2}}
        +
        \frac{G_\y R_{\cY}}{2(T+1)^{1/2}}.
    \end{align*}
\end{theorem}

\begin{proof}[Proof of \cref{thm:lower-bound-lmo-po-repeat}]
    Fix $T\geq 1$, and choose dimensions $n,m$ such that $n,m\geq 4(T+1)$. Let
    \begin{align*}
        \cX:=R_{\cX} \cdot \Delta_n.
    \end{align*}
    The adversary initializes the primal variable at the extreme point $x_1=R_{\cX} \cdot e_1$ and fixes an LMO tie-breaking rule that always returns an extreme-point minimizer. 
    For the dual variable, we invoke the classical nonsmooth first-order lower bound over the Euclidean ball. Namely, for any adaptive first-order method making $T$ oracle queries and returning $\bar y_T$, there exist a compact convex set $\cY:=R_{\cY} \cdot \mathbb{B}_2^m$ and a convex $G_\y$-Lipschitz function $h_{T+1}:\cY\to\R$ such that
    \begin{align}
    \label{eq:po-lower-bound-y}
        h_{T+1}(\bar y_T)-\min_{y\in\cY} h_{T+1}(y)
        \geq
        \frac{G_\y R_{\cY}}{2(T+1)^{1/2}}.
    \end{align}
    This is the standard nonsmooth lower-bound construction of \citet{nesterov2013introductory}. For completeness, we repeat the construction here.
    Set
    \begin{align*}
        \delta:=\frac{R_{\cY}}{4T(T+1)^{1/2}}.
    \end{align*}
    We inductively construct signs $a_1,\ldots,a_{T+1}\in\{\pm 1\}$ and distinct coordinates $\sigma(1),\ldots,\sigma(T+1)\in \{1, \dots, m\}$. After $y_k$ has been determined, choose
    \begin{align*}
        \sigma(k)\in\argmax_{j\in[m]\setminus\{\sigma(1),\ldots,\sigma(k-1)\}} |(y_k)_j|,
    \end{align*}
    and choose $a_k\in\{\pm 1\}$ so that
    \begin{align*}
        a_k(y_k)_{\sigma(k)}=|(y_k)_{\sigma(k)}|.
    \end{align*}
    For $k=1,\ldots,T+1$, define
    \begin{align*}
        h_k(y):=
        G_\y\max_{1\leq i\leq k}
        \left\{
            a_i y_{\sigma(i)}-2(i-1)\delta
        \right\}.
    \end{align*}
    We claim that the final function $h_{T+1}$ is consistent with all information revealed to the algorithm during the first $T$ iterations. To see this, fix $1\leq s<k\leq T+1$ and any $y\in\cY$ such that $\|y-y_s\|_2\leq\delta$. For every $i>s$, the coordinate $\sigma(i)$ was still unused when $\sigma(s)$ was chosen. Hence
    \begin{align*}
        |(y_s)_{\sigma(s)}|\geq |(y_s)_{\sigma(i)}|\geq a_i(y_s)_{\sigma(i)}.
    \end{align*}
    Therefore,
    \begin{align*}
        a_i y_{\sigma(i)}-2(i-1)\delta
        &\leq a_i(y_s)_{\sigma(i)}+\delta-2(i-1)\delta \\
        &\leq |(y_s)_{\sigma(s)}|+\delta-2s\delta \\
        &= |(y_s)_{\sigma(s)}|-(2s-1)\delta \\
        &\leq a_s y_{\sigma(s)}-2(s-1)\delta.
    \end{align*}
    Thus every affine piece added after step $s$ is dominated on the ball $\{y:\|y-y_s\|_2\leq\delta\}$ by the $s$-th affine piece, which already appears in $h_s$. Consequently,
    \begin{align*}
        h_k(y)=h_s(y),
        \qquad
        \forall y\in\cY \text{ such that } \|y-y_s\|_2\leq\delta.
    \end{align*}
    In particular, the values and subdifferentials observed at each previous query $y_s$, $s\leq T$, are identical for $h_s$ and for the final function $h_{T+1}$. 
    The function $h_{T+1}$ is convex and $G_\y$-Lipschitz on $\cY$, since it is the maximum of affine functions whose slopes have Euclidean norm $G_\y$. Moreover,
    \begin{align*}
        \min_{y\in\cY}h_{T+1}(y)
        &\leq
        G_\y\min_{\|y\|_2\leq R_{\cY}}
        \max_{1\leq i\leq T+1}
        a_i y_{\sigma(i)} 
        =
        -\frac{G_\y R_{\cY}}{(T+1)^{1/2}}.
    \end{align*}
    On the other hand, since $\bar y_T=y_{T+1}$ in \cref{alg:generic-lmo-po}, we have
    \begin{align*}
        h_{T+1}(y_{T+1})
        \geq
        G_\y\left(a_{T+1}(y_{T+1})_{\sigma(T+1)}-2T\delta\right) 
        =
        G_\y\left(|(y_{T+1})_{\sigma(T+1)}|-2T\delta\right) 
        \geq
        -2G_\y T\delta.
    \end{align*}
    Therefore, by the choice of $\delta$,
    \begin{align*}
        h(\bar y_T)-\min_{y\in\cY}h(y)
        \geq
        \frac{G_\y R_{\cY}}{(T+1)^{1/2}}-2G_\y T\delta
        =
        \frac{G_\y R_{\cY}}{2(T+1)^{1/2}}.
    \end{align*}
    
    Define now the saddle payoff
    \begin{align*}
        f(x,y):=G_\x\|x\|_2-h_{T+1}(y).
    \end{align*}
    Since $x_1$ is an extreme point of $\cX$ and each primal LMO call returns an extreme point of $\cX$, all points $x_1,\ldots,x_{T+1}$ are vertices of $R_{\cX}\Delta_n$. By \cref{alg:generic-lmo-po}, the output satisfies
    \begin{align*}
        \bar x_T\in R_{\cX} \cdot \conv\{e_{i_1},\ldots,e_{i_{k_T}}\}.
    \end{align*}
    Consequently, we have
    \begin{align*}
        \|\bar x_T\|_2
        \geq
        \min\left\{\|x\|_2:x\in R_{\cX} \cdot \conv\{e_{i_1},\ldots,e_{i_{k_T}}\}\right\} 
        =
        \frac{R_{\cX}}{k_T^{1/2}}
        \geq
        \frac{R_{\cX}}{(T+1)^{1/2}}.
    \end{align*}
    Since $n\geq 4(T+1)$, we may conclude that
    \begin{align}
    \label{eq:lmo-po-primal-lower}
        G_\x\|\bar x_T\|_2-G_\x\min_{x\in\cX}\|x\|_2
        \geq
        \frac{G_\x R_{\cX}}{2(T+1)^{1/2}}.
    \end{align}
    Finally, by the separable form of $f$,
    \begin{align*}
        \max_{y\in\cY}f(\bar x_T,y)-\min_{x\in\cX}f(x,\bar y_T) 
        &=
        G_\x\|\bar x_T\|_2-\min_{y\in\cY}h_{T+1}(y)
        -
        \left(
            G_\x\min_{x\in\cX}\|x\|_2-h_{T+1}(\bar y_T)
        \right) \\
        &=
        G_\x\|\bar x_T\|_2-G_\x\min_{x\in\cX}\|x\|_2
        +
        h_{T+1}(\bar y_T)-\min_{y\in\cY}h_{T+1}(y).
    \end{align*}
    Combining this identity with~\eqref{eq:po-lower-bound-y} and~\eqref{eq:lmo-po-primal-lower} yields the desired bound.
\end{proof}

\section{Implementation Details}
\label{app:numerics}
This appendix collects the implementation details used in \cref{sec:numerical}. In all experiments, the suffix \texttt{Deterministic} means that the method uses the full subgradient oracle, while the suffix \texttt{Stochastic} means that the method uses an unbiased stochastic estimate of the same oracle. The algorithm names \texttt{LMO-LMO}, \texttt{LMO-PO}, \texttt{PO-LMO}, and \texttt{PO-PO} refer to the oracle used for the primal and dual domains, respectively. They correspond to \cref{alg:lmo-lmo,alg:lmo-po,alg:po-lmo,alg:po-po}, respectively.
For LMO-based updates, we use simple enclosing sets $\myoverline{\cX}\supseteq\cX$ and $\myoverline{\cY}\supseteq\cY$. These sets are chosen so that projection onto them is computationally cheap, while the actual feasible updates over $\cX$ and $\cY$ are performed using LMOs whenever the corresponding method uses an LMO. 
For the robust classification problem, we rely on the following set $\myoverline{\cY}$ for the inexpensive projection step.

\begin{lemma}
\label{lem:proj-oracle}
    Given any $c\in\R^n$, let $[c]_+=\max\{c,0\}$ componentwise. For
    \begin{align*}
        \myoverline{\cY}:=\{y\in\R^n:y\geq 0,\|y\|_2\leq 1\},
    \end{align*}
    it holds that
    \begin{align*}
        \PO_{\myoverline{\cY}}(c)=\frac{[c]_+}{\max\{\|[c]_+\|_2,1\}}.
    \end{align*}
\end{lemma}
\begin{proof}[Proof of \cref{lem:proj-oracle}]
    Computing $\PO_{\myoverline{\cY}}(c)$ is equivalent to solving
    \begin{align*}
        \min_{y\in\R^n} \{ \|y-c\|_2^2 : y \geq 0,\ \|y\|_2 \leq 1 \}.
    \end{align*}
    Since the feasible set is the intersection of the nonnegative orthant and the Euclidean unit ball, the negative coordinates of $c$ are first clipped to zero. Thus the problem reduces to projecting $[c]_+$ onto the Euclidean unit ball. If $\|[c]_+\|_2\leq 1$, then $[c]_+$ is feasible and is the projection. If $\|[c]_+\|_2>1$, the Euclidean projection onto the unit ball is $[c]_+/\|[c]_+\|_2$. Combining the two cases gives
    \begin{align*}
        \PO_{\myoverline{\cY}}(c)=\frac{[c]_+}{\max\{\|[c]_+\|_2,1\}}.
    \end{align*}
    This proves the claim.
\end{proof}

The projection and LMO formulas used in the experiments are summarized in \cref{table:oracles_summary}, where the input is the vector $c$ or the matrix $C$ depending on the context.

\begin{table}[!t]
    \centering
    \caption{Summary of LMOs and POs.}
    \label{table:oracles_summary}
    \small
    \renewcommand{\arraystretch}{1.5}
    \begin{minipage}{0.875\textwidth}
    \begin{tabularx}{\textwidth}{@{} l l X @{}}
        \toprule
        \textbf{Set } & \textbf{LMO} & \textbf{PO}  \\
        \midrule
        \textbf{Frobenius Ball} & $-r \frac{C}{\|C\|_2}$ & If $\|C\|_2 > r$, $Y = r \frac{C}{\|C\|_{\mathrm{F}}}$. \\
        ($\|Y\|_{\mathrm{F}} \le r$) & & Else $Y = C$. \\
        \addlinespace
        \textbf{Nuclear Ball} & $-r(u_1 v_1^T)$ & 1. Full SVD: $C = U\Sigma V^T$. \\
        ($\|Y\|_* \le r$) & (Top singular triplet) & 2. Project $\text{diag}(\Sigma)$ onto $r$-Simplex. \\
        & & 3. Reconstruct: $Y = U \Sigma_{\text{proj}} V^T$. \\
        \addlinespace
        \textbf{Nonnegative Ball} & --- & 1. Nonnegative clip: $[c]_+ = \max(0, c)$ \\
        ($\|y\|_2 \le r, y \ge 0$) & & 2. If $\|[c]_+\|_2 > r$, $Y = r \frac{[c]_+}{\|[c]_+\|_2}$. Else $y = [c]_+$. \\
        \addlinespace
        \textbf{$r$-Simplex}        & $-r\text{sign}(c_i) \cdot e_i$          & 1. Sort $C$ in descending order. \\
        ($\sum y_i = r, y \ge 0$) & where $i = \text{argmin}(c)$ & 2. Find $\tau$ s.t. $\sum \max(c_i - \tau, 0) = r$. \\
        & & 3. Output: $y_i = \max(c_i - \tau, 0)$. \\
        \bottomrule
    \end{tabularx}
    \end{minipage}
\end{table}

\subsection{Matrix Completion with Spectral Norm Fit}
\citet{juditsky2016solving} considered the spectral norm fit problem
\begin{equation} \label{eq:MC}
    \min_{\|X\|_* \leq 1} \|\cA(X)-B\|_{\text{op}},
\end{equation}
where $X \in \R^{n \times p}$ is the decision variable, $B \in \R^{m \times q}$, and $\cA:\R^{n \times p}\to\R^{m \times q}$ is a linear map. Problem~\eqref{eq:MC} admits the saddle-point reformulation
\begin{align*}
    \min_{\|X\|_* \leq 1} ~ \max_{\|Y\|_* \leq 1} ~
    \Tr\left[(\cA(X)-B)^\top Y\right].
\end{align*}

For the matrix completion with spectral norm fit problem, we generate a synthetic instance as suggested by \citep{juditsky2016solving} as follows: We set $p=n$ and $q=m=2n$ with $n = 200$ and the map $\cA$ given by
$$\cA(X)= \frac{1}{k} \sum_{i =1}^{k} L_i X R_i^\top,$$
where $L_i \in \R^{m \times n}$ and $R_i \in \R^{q \times p}$ with $L_i = \Tilde{L}_i/\| \Tilde{L}_i \|_{\text{op}}$ and $R_i = \Tilde{R}_i/\| \Tilde{R}_i \|_{\text{op}}$ and $\Tilde{L}_i, \Tilde{R}_i$ are randomly generated by standard normal distributions. In our experiment, we set $k = 2$. The construction of the observation matrix $B$ begins with the generation of a ground-truth signal $V_{\text{true}} \in \mathbb{R}^{n \times n}$. To ensure this matrix is low-rank, we first define the rank as $r = \lfloor \sqrt{n} \rfloor$. We then sample two latent factor matrices $U, V \in \mathbb{R}^{n \times r}$ from a standard normal distribution and form the base signal matrix through the outer product $V_{\text{base}} = UV^\top$.\\

To construct $B$, we first compute $V_{\text{true}} = V_{\text{base}} / \|V_{\text{base}}\|_*$, which ensures that the resulting ground-truth matrix has a unit nuclear norm. This normalized signal is then passed through a linear observation operator $\mathcal{A}(\cdot)$ to produce the clean measurement matrix $B_{\text{clean}} = \mathcal{A}(V_{\text{true}}) \in \mathbb{R}^{m \times m}$. We then generate a raw noise matrix $E_{\text{raw}} \in \mathbb{R}^{m \times m}$ where each entry is sampled from standard normal distribution. This matrix is then rescaled by its spectral norm and multiplied by a specified noise level $\delta$ to form the error matrix $E = \delta \cdot (E_{\text{raw}} / \|E_{\text{raw}}\|_{\text{op}})$. The observed matrix $B$ is then obtained by the summation of the clean measurements and this spectrally-scaled noise, $B = B_{\text{clean}} + E$.\\

For initialization, we start with $X_0 = \PO_{\cX}( \Tilde{X}_0)$ and $Y_0 = \PO_{\cY}(\Tilde{Y}_0)$ where entries of $\Tilde{X}_0$ and $\Tilde{Y}_0$ are randomly generated by uniform distribution over $[-1000,1000]$. For \cref{alg:lmo-lmo}, \cref{alg:lmo-po}, \cref{alg:po-lmo} and \cref{alg:po-po}, the stepsizes are given by 
$$\alpha_t =\beta_t= \left(\frac{1}{t+1}\right)^{-1/2} \quad \& \quad \eta_t = \tau_t =  \frac{0.1}{(t+1)^{1/2}} \quad \& \quad  \gamma_t=\rho_t =\left(\frac{1}{t+1}\right)^{1/2},$$
which align with \cref{theorem:convergence} and \cref{theorem:convergence-po-po-2}.
For $\cX = \{ X \in \R^{n \times p} \mid \| X \|_* \leq 1 \}$, we choose the covering $\overline{\cX} = \{ X \in \R^{n \times p} \mid \| X \|_{\mathrm{F}} \leq 1 \}$. For $\cY = \{ X \in \R^{m \times q} \mid \| Y \|_* \leq 1 \}$, we choose the covering $\overline{\cY} = \{ Y \in \R^{m \times q} \mid \| Y \|_{\mathrm{F}} \leq 1 \}$.\\

Moreover, since the payoff function in this experiment is smooth, we also implement LMO-LMO, LMO-PO and PO-LMO methods \citep{giang2026projection},  Alternating Gradient Projection (AGP) method \citep{xu2023unified}, ExtraGradient (EG) method \citep{korpelevich1976extragradient}, Optimistic Gradient Descent-Ascent (OGDA) method \citep{popov1980modification}. To describe the parameter choices, we follow the notation adopted in the corresponding paper. For LMO-LMO method \citep[Algorithm 1]{giang2026projection}, we choose $\tau_t = (t+1)^{-1}$ and $\beta_t = (t+1)^{-1/5}$ according to \citep[Theorem 2.2]{giang2026projection}. For LMO-PO method \citep[Algorithm 2]{giang2026projection}, we choose $\tau_t = (t+1)^{-1}$ and $\beta_t = (t+1)^{-1/3}$ according to \citep[Theorem 3.2]{giang2026projection}. For PO-LMO method \citep[Algorithm 3]{giang2026projection}, we choose $\tau_t = (t+1)^{-2/3}$ and $\beta_t = (t+1)^{-1/3}$ according to \citep[Theorem 4.2]{giang2026projection}. For AGP method \citep{xu2023unified}, we implement both parameter choices in \citep[Theorem 3.2 ]{xu2023unified}, i.e., $b_t = 0, c_t = (t+1)^{-1/4}$, $\beta_t =(t+1)^{1/2},\gamma_t = 10$ and in \citep[Theorem 4.2]{xu2023unified}, i.e., $b_t = (t+1)^{-1/4}, c_t = 0, \beta_t = 10, \gamma_k = (t+1)^{1/2}$ which correspond to parameter choices that ensure convergence in nonconvex-concave and convex-nonconcave settings, respectively. For EG method \citep{korpelevich1976extragradient}, we choose the stepsize to be $\frac{1}{2L}$ and for OGDA method \citep{popov1980modification}, we choose the stepsize to be $\frac{1}{2.05L}$ where $L = \frac{1}{k}\sum_{i = 1}^k \|L_i\|_{\text{op}}\|R_i\|_{\text{op}}$ is the smoothness constant of the payoff function.

Figure~\ref{fig:FullExperiment} provides the extended comparison on the smooth matrix completion instance. As expected for this bilinear problem, the classical smooth methods, especially OGDA and EG, perform very well when equipped with tuned stepsizes. Among the proposed nonsmoothing methods, \texttt{LMO-LMO (Det)} and \texttt{LMO-LMO (Sto)} are the strongest variants, and they substantially outperform the smoothing-based projection-free methods. The smoothing-based LMO variants decrease much more slowly and remain close to their initial saddle-gap values over the plotted time horizon, reflecting the cost of introducing and tuning a smoothing parameter. The AGP variants also improve more slowly than EG/OGDA and the proposed LMO-based methods. Overall, the figure illustrates that, even on a smooth problem where methods designed specifically for smooth saddle-point problems are highly competitive, the proposed subgradient-based LMO methods remain effective and avoid the deterioration observed for smoothing-based projection-free schemes.

\begin{figure*}[!b]
    \centering

    \begin{subfigure}[b]{1\textwidth}
        \centering
        \includegraphics[width =0.8\linewidth]{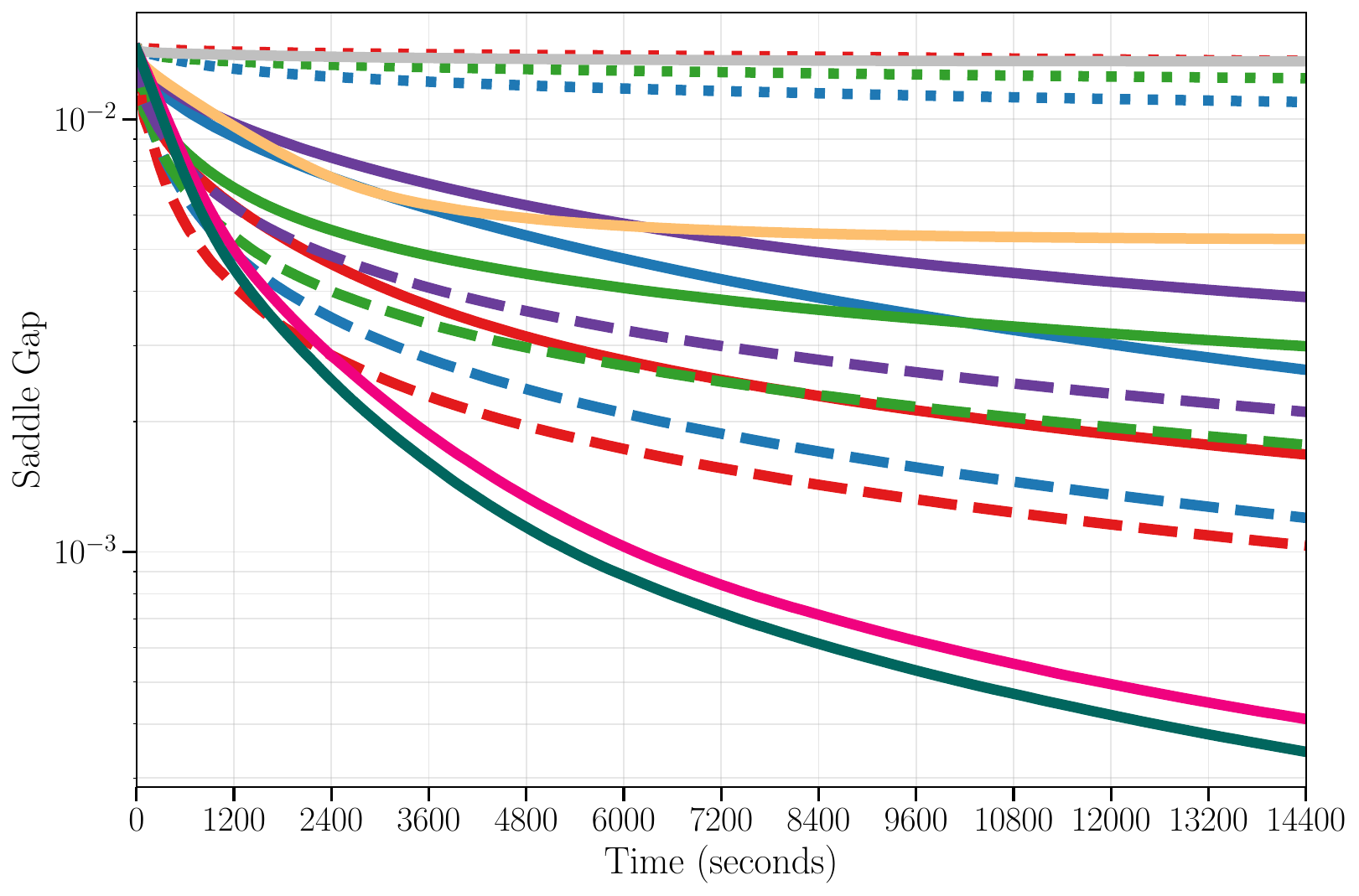}
        \label{subfig:fullmatrix}
    \end{subfigure}

    \vspace{0.2em} 
    \centering
    \begin{subfigure}[b]{1\textwidth}
        \centering
        \includegraphics[width=0.65\linewidth]{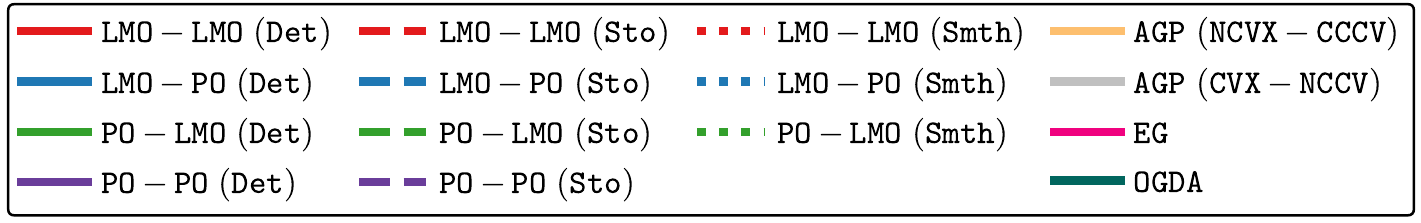}
    \end{subfigure}
    \caption{Additional comparison on the matrix completion problem with spectral norm fit. %Since this problem is smooth, we include smoothing-based projection-free methods, a smoothing-based projection method, extra-gradient, and optimistic gradient descent-ascent.
    }
    \label{fig:FullExperiment}
\end{figure*}

\subsection{Robust Multiclass Classification with Hinge Loss}
We use the robust multiclass classification model described in \cref{sec:numerical}. In this experiment, we set
\begin{align*}
    r=10,
    \qquad
    \lambda=\frac{\lambda'}{n^2},
    \qquad
    \lambda'=1.
\end{align*}
The primal and dual feasible sets are
\begin{align*}
    \cX=\{\Theta\in\R^{k\times d}:\|\Theta\|_*\leq r\}
    \quad \text{and} \quad
    \cY=\left\{y\in\R^n:y\geq 0,\sum_{i=1}^{n}y_i=1\right\}.
\end{align*}
For the LMO-based variants, we use the covering sets
\begin{align*}
    \myoverline{\cX}=\{\Theta\in\R^{k\times d}:\|\Theta\|_{\mathrm{F}}\leq r\},
    \quad \text{and} \quad
    \myoverline{\cY}=\{y\in\R^n:y\geq 0,\|y\|_2\leq 1\}.
\end{align*}
The inclusion $\cY\subseteq\myoverline{\cY}$ follows because every probability vector has Euclidean norm at most one. Projection onto $\myoverline{\cX}$ is a Frobenius-norm rescaling, while projection onto $\myoverline{\cY}$ is given in \cref{lem:proj-oracle}.
We initialize $\Theta_1=\PO_{\cX}(\widetilde\Theta_1)$ and $y_1=e_1$, where the entries of $\widetilde \Theta_1$ are sampled independently from the uniform distribution on $[-1000,1000]$. The stepsizes are
\begin{align*}
    \alpha_t=\beta_t=\left(\frac{0.01}{t+1}\right)^{-1/2},
    \qquad
    \eta_t=\tau_t=\left(\frac{10}{t+1}\right)^{1/2},
    \qquad
    \gamma_t=\rho_t=\left(\frac{0.01}{t+1}\right)^{1/2}.
\end{align*}
These choices follow the scaling in \cref{theorem:convergence,theorem:convergence-po-po-2}.
Let
\begin{align*}
    \ell_i(\Theta)
    =
    \max_{j\in[k]}
    \left\{
        \one(j\neq b_i)+\theta_j^\top a_i-\theta_{b_i}^\top a_i
    \right\}.
\end{align*}
The full primal subgradient is computed from $\frac{1}{n} \sum_{i=1}^{n}y_i \partial \ell_i(\Theta).$
The dual oracle uses the subgradient of $(-f)_\Theta(\cdot)$, namely $-\frac{1}{n}\ell(\Theta)+2\lambda n(ny-\one_n),$ where $\ell(\Theta) = (\ell_1(\Theta),\ldots,\ell_n(\Theta))^\top$. The stochastic variants replace the full sums by unbiased single-sample estimates.
For this experiment, we report the primal objective
\begin{align*}
    p(\Theta)
    :=
    \max_{y\in\cY}
    \left\{
        \frac{1}{n}\sum_{i=1}^{n}y_i\ell_i(\Theta)
        -
        \lambda\left\|ny-\one_n\right\|_2^2
    \right\}.
\end{align*}
Computing the full strong saddle gap is expensive at this scale because it requires solving both the dual maximization problem and the primal minimization problem over the nuclear-norm ball. The primal objective is substantially cheaper to compute and empirically tracks the same convergence behavior.

%\newpage
%\input{checklist.tex}

\end{document}